\documentclass[final,onefignum,onetabnum]{siamart251216}
\usepackage{booktabs}
\usepackage{amssymb,amsfonts,mathtools}
\usepackage{bm}
\usepackage{mathrsfs}
\usepackage{float}
\usepackage{algpseudocode}
\usepackage[utf8]{inputenc}
\usepackage[export]{adjustbox}
\usepackage{wrapfig}
\usepackage{multirow}
\usepackage[normalem]{ulem}
\usepackage{tikz}
\usepackage{xcolor}
\usepackage[title]{appendix}
\hypersetup{hypertexnames=false}

\theoremstyle{nonumberplain}
\theoremheaderfont{\normalfont\itshape}
\theorembodyfont{\normalfont}
\theoremseparator{.}
\theoremsymbol{}
\newtheorem{remark}{Remark}

\theoremstyle{plain}
\theoremheaderfont{\normalfont\sc}
\theorembodyfont{\normalfont\itshape}
\theoremseparator{.}
\theoremsymbol{}
\newtheorem{example}{Example}[section]
\newtheorem{assumption}{Assumption}
\newcommand{\LT}{\mathscr{L}}
\newcommand{\Nn}{\mathcal{N}}
\newcommand{\Aa}{\mathcal{A}}
\newcommand{\Ii}{\mathcal{I}}
\newcommand{\Rr}{\mathcal{R}}
\newcommand{\FT}{\mathscr{F}}

\newcommand{\Log}{\text{Log}}
\title{Convolution Quadrature with Adaptive History Compression for Variable-Order Fractional Reaction--Diffusion Equations}
\author{Pu Yuan\thanks{Mathematical Institute, Utrecht University, The Netherlands (\email{p.yuan@uu.nl}).}
\and Paul Zegeling\thanks{Mathematical Institute, Utrecht University, The Netherlands (\email{p.a.zegeling@uu.nl}).}}

\date{}

\begin{document}

\maketitle

\begin{abstract}
We develop a fully implicit fast method for semilinear variable-order fractional reaction--diffusion equations with periodic fractional diffusion and exponential-transition Scarpi memory.
A Volterra-resolvent analysis establishes local well-posedness and temporal regularity estimates that describe the weak initial singularity. 
Backward Euler convolution quadrature combined with Fourier pseudospectral discretization achieves mesh-uniform stability and positive-time first-order convergence.
To accelerate the scheme, we utilize Cauchy--FFT weight recovery and adaptive blockwise sum-of-exponentials compression, which significantly reduces the cost of history evaluation under local error control.
A logarithmic perturbation bound quantifies the effects of weight recovery, history compression, and inexact nonlinear solves, giving tolerance conditions that preserve the convergence order.
Numerical experiments confirm the first-order accuracy and demonstrate that the accelerated method efficiently resolves long-time spatial patterns at a substantially reduced computational cost.
\end{abstract}


\section{Introduction}
Time-fractional evolution equations model diffusion and relaxation processes with memory, while variable-order models introduce a time-dependent order function $\alpha(t)$ to capture transitions between different memory regimes \cite{LorenzoHartley2002,GarrappaGiustiMainardi2021}. 
We adopt Scarpi's formulation in the Laplace domain because it preserves a convolution structure in time \cite{Scarpi1972b,GarrappaGiustiMainardi2021}. Specifically, we consider the semilinear problem
\begin{equation}
\label{eq:intro-model}
\begin{cases}
{}^{S}D_0^{\alpha(t)}u(x,t)
=\Aa u(x,t)+\Nn(u(x,t)),
& x\in\Omega,\quad 0<t\le T,\\
u(\cdot,t)\ \text{is periodic on }\Omega,
\qquad
u(\cdot,0)=u_0,
\end{cases}
\end{equation}
where ${}^{S}D_0^{\alpha(t)}$ denotes the Scarpi time derivative
with order function
\[
\alpha(t)=\alpha_2+(\alpha_1-\alpha_2)e^{-ct},
\qquad
0<\alpha_1,\alpha_2<1,\quad c>0.
\]
The diffusion operator is $\Aa=-d_u(-\Delta)^{\beta/2}$,
with $d_u>0$, where $\beta=2$ gives classical diffusion
and $0<\beta<2$ gives nonlocal diffusion on the torus
\cite{RoncalStinga2016}.
The reaction term $\Nn$ may be nonlinear, stiff, and vector-valued.

The model \eqref{eq:intro-model} can be reformulated as a Volterra integral equation with a convolution kernel whose Laplace transform $\Psi_\alpha$ is available explicitly.
Convolution quadrature (CQ) is therefore particularly suitable for this integral formulation, since its weights can be generated directly from $\Psi_\alpha$
\cite{Lubich1986,Lubich1988I,Lubich1988II}.
Garrappa and Giusti \cite{GarrappaGiusti2023} apply backward Euler CQ to exponential-type Scarpi equations and compute the weights by Cauchy quadrature and FFT.
However, direct CQ history summation requires $\mathcal O(N^2)$ work over $N$ time steps.
Fast and oblivious CQ and sum-of-exponentials (SOE) techniques address this cost for various convolution kernels \cite{SchadleLopezFernandezLubich2006,jiang2017fast,zhu2019fast}.
For the present problem, two issues must be treated simultaneously: the Scarpi weights are generated from a nonalgebraic Laplace symbol, and the semilinear Fourier discretization requires perturbation estimates that remain uniform under refinement of the spatial grid.
The latter estimates are established through discrete resolvent bounds.

We combine backward Euler CQ (BE--CQ) with Fourier pseudospectral discretization and accelerate only the history evaluation.
Cauchy--FFT recovery provides the CQ weights.
Recent recovered weights are used without SOE approximation, whereas older weights are fitted on adaptive lag blocks with inherited local error tolerances.
The partition and local ranks adapt to different lag scales, and the local fitting residuals, together with a checked right-hand-side bound, control the total history-compression defect.
The fully implicit current-step equation is unchanged and is solved by damped inexact Newton--GMRES with reaction-dependent preconditioning and reuse of Fourier transforms.

We first use the Sonine structure of the Scarpi kernel to reformulate \eqref{eq:intro-model} as a Volterra integral equation and establish the well-posedness of the continuous semilinear problem, and derive temporal regularity estimates needed for the CQ error analysis.
For the fully implicit BE--CQ scheme, discrete resolvent estimates yield stability constants independent of the Fourier grid.
Under the stated temporal regularity and Lipschitz assumptions, the full-history scheme is first-order accurate at every positive distance from the initial time:

\[
\max_{t_*\le t_n\le T}
\|u^n-u(t_n)\|_X
\le C_*h,
\qquad t_*>0.
\]
For the accelerated computation, the effects of CQ-weight recovery, blockwise SOE compression, and inexact nonlinear solution are separated.
If $U_{\rm e}$ denotes the exact-weight, full-history discrete solution and $U_{\rm num}$ the computed accelerated solution, the stability transfer gives
\[
\max_{0\le n\le N}
\|U_{\rm num}^n-U_{\rm e}^n\|_{\boldsymbol M}
\le
C_T(1+\log(N+1))
\bigl(
\tau_{\rm hist}
+\varepsilon_{\rm weight}
+\varepsilon_{\rm alg}
\bigr),
\]
under the stated assumptions and history-defect condition, with $C_T$ independent of $h$ and the Fourier grid, and $\|\cdot\|_{\boldsymbol M}$ denotes the discrete $L^2$ norm.
Choosing $\tau_{\rm hist}+\varepsilon_{\rm weight}+\varepsilon_{\rm alg}
=\mathcal O(h/[1+\log(N+1)])$ preserves positive-time first-order accuracy.

The numerical experiments examine a stiff reaction system and two-dimensional Nagumo and FitzHugh--Nagumo problems.
The stiff reaction test shows first-order temporal convergence and illustrates the advantage of blockwise fitting over the tested single-tail fits in meeting the prescribed residual tolerance.
The experiments demonstrate substantial computational savings while accurately approximating the full-history solutions. The PDE simulations exhibit comparable Newton--Krylov iteration counts. Additionally, we assess the sharpness of the history-defect bounds and verify the compression and right-hand-side acceptance conditions.

The paper is organized as follows.
Section~\ref{sec:preliminaries} introduces the Scarpi kernels and establishes the symbol and resolvent estimates.
Section~\ref{sec:well-posedness} studies the well-posedness and temporal regularity of the Scarpi--Volterra problem.
Section~\ref{sec:be-cq} develops the BE--CQ discretization, proves stability and temporal error estimates, and describes the Fourier pseudospectral solver.
Section~\ref{sec:adaptive-block-soe} presents the adaptive blockwise SOE method and analyzes its perturbation error and computational cost.
Section~\ref{sec:numerics} reports the numerical experiments, and we discuss the results and future directions in the conclusion section.

\section{Preliminaries}
\label{sec:preliminaries}
Let $\Omega=\prod_{j=1}^d(0,L_j)$ be a periodic domain.
For a periodic function $v$, let $\FT v$ denote its Fourier coefficients and set
\[
\boldsymbol\kappa_{\boldsymbol k}
=2\pi\left(\frac{k_1}{L_1},\ldots,\frac{k_d}{L_d}\right),
\qquad \boldsymbol k\in\mathbb Z^d.
\]
The periodic fractional Laplacian is the Fourier multiplier
\[\bigl(\FT[(-\Delta)^{\beta/2}v]\bigr)_{\boldsymbol k}
=|\boldsymbol\kappa_{\boldsymbol k}|^\beta(\FT v)_{\boldsymbol k},
\qquad 0<\beta\le2,\]
see \cite{RoncalStinga2016}.

\begin{assumption}
\label{ass:scarpi-resolvent}
Let $X=H^\sigma_{\mathrm{per}}(\Omega;\mathbb C)$, with $\sigma>d/2$, and let $\Aa=-d_u(-\Delta)^{\beta/2}$, where $d_u>0$ and $D(\Aa)=H^{\sigma+\beta}_{\mathrm{per}}(\Omega;\mathbb C)$.
The nonlinearity $\Nn:X\to X$ is Lipschitz on bounded subsets of $X$ and preserves the real subspace whenever real-valued solutions are considered.
\end{assumption}

The space $X$ is a Banach algebra, and $\Aa$ is densely defined and closed.
We equip $D(\Aa)$ with its graph norm and write $\Ii$ for the identity on $X$.

\begin{lemma}
\label{lem:fourier-resolvent}
For every \(\eta\in(0,\pi)\) and every \(\lambda\ne0\) satisfying \(|\operatorname{Arg}\lambda|\le\pi-\eta\), we have \(\lambda\in\rho(\Aa)\) and
\begin{equation}
\label{eq:fourier-resolvent-bound}
\|(\lambda\Ii-\Aa)^{-1}\|_{\mathcal L(X)}
\le\frac{1}{\sin\eta}\,|\lambda|^{-1}.
\end{equation}
The same estimate holds for every Fourier truncation of \(\Aa\), independently of the grid.
\end{lemma}

\begin{proof}
For \(s\ge0\), the angular restriction gives \(|\lambda+s|\ge(\sin\eta)|\lambda|\).
The Fourier multiplier
\[
\bigl(\FT[(\lambda\Ii-\Aa)^{-1}v]\bigr)_{\boldsymbol k}
=\frac{(\FT v)_{\boldsymbol k}}
{\lambda+d_u|\boldsymbol\kappa_{\boldsymbol k}|^\beta}
\]
therefore maps \(X\) into \(D(\Aa)\) and satisfies \eqref{eq:fourier-resolvent-bound}.
This proof remains valid upon restricting the set of Fourier modes.
\end{proof}

We next recall the framework of variable-order (VO) fractional calculus adopted in this work, based on the Laplace-domain formulation of Scarpi-type operators \cite{Scarpi1972b,GarrappaGiustiMainardi2021,GarrappaGiusti2023}.
\begin{definition}
Let $\alpha\in(0,1)$ be a constant.
The Riemann--Liouville fractional integral and the Caputo fractional derivative are defined by
\[\begin{aligned}
(I_0^{\alpha} f)(t)
&= \int_0^t \psi(t-\tau) f(\tau)\,d\tau,
\qquad
\psi(t)=\frac{t^{\alpha-1}}{\Gamma(\alpha)}, \\
(D_0^{\alpha} f)(t)
&= \int_0^t \phi(t-\tau) f'(\tau)\,d\tau,
\qquad
\phi(t)=\frac{t^{-\alpha}}{\Gamma(1-\alpha)}.
\end{aligned}\]
\end{definition}
The Laplace transforms of the kernels are
\[
\LT\{\psi\}(z)=z^{-\alpha},
\qquad
\LT\{\phi\}(z)=z^{\alpha-1},
\]
and hence $\LT\{\phi*\psi\}=z^{-1}$.
Therefore, $(\psi,\phi)$ form a Sonine pair, and this provides the convolutional inverse structure underlying the fractional integral and derivative.
Consequently,
\[
D_0^{\alpha} I_0^{\alpha} f = f,
\qquad
I_0^{\alpha} D_0^{\alpha} f = f - f(0),
\]
under suitable regularity assumptions \cite{Luchko2020,SamkoKilbasMarichev1993}.

Based on the Sonine relation, we now introduce the variable-order framework.
Following \cite{GarrappaGiusti2023}, Scarpi's construction may be formulated for a locally integrable function \(\alpha:[0,\infty)\to[0,1]\) whose Laplace transform is known in a right half-plane.
In this work we restrict throughout to the continuous exponential transition.

\begin{definition}
Let $\alpha(t)=\alpha_2+(\alpha_1-\alpha_2)e^{-ct}$ with $0<\alpha_1,\alpha_2<1$ and $c>0$.
Its Laplace transform and order symbol are
\[
\widehat\alpha(z)=\LT\{\alpha\}(z)
=\frac{\alpha_2c+\alpha_1z}{z(c+z)},
\qquad
B(z)=z\widehat\alpha(z)=\frac{\alpha_2c+\alpha_1z}{c+z},
\qquad \Re z>0.
\]
On \(\mathbb C\setminus(-\infty,0]\), use the rational continuation of \(B\) and the principal logarithm to define
\[
\Psi_\alpha(z)=z^{-B(z)},
\qquad
\Phi_\alpha(z)=z^{B(z)-1}=\frac{1}{z\Psi_\alpha(z)},
\]
and
\begin{equation}
\label{eq:inverse_LT_kernel}
\psi_\alpha(t)=\LT^{-1}\{\Psi_\alpha\}(t),
\qquad
\phi_\alpha(t)=\LT^{-1}\{\Phi_\alpha\}(t).
\end{equation}
For \(T>0\), \(f\in L^1(0,T;X)\), and \(u\in W^{1,1}(0,T;X)\), set
\[
{}^{S}I_0^{\alpha(t)}f=\psi_\alpha*f,
\qquad
{}^{S}D_0^{\alpha(t)}u=\phi_\alpha*u',
\]
where \((k*f)(t)=\int_0^t k(t-s)f(s)\,ds\).
\end{definition}

To show that $\psi_\alpha$ and $\phi_\alpha$ are well defined and satisfy the Sonine relation, we choose contours that avoid the singular point \(z=-c\), which is the pole of \(B\), and exclude the branch cut of the principal logarithm.
For \(\vartheta\in(0,\pi)\), set
\[
\Sigma_\vartheta
=\{z\in\mathbb C\setminus\{0\}:|\operatorname{Arg}z|<\vartheta\}.
\]
For \(\theta\in(\pi/2,\pi)\) and \(\kappa>0\), define the upward-oriented contour
\[
\Gamma_{\theta,\kappa}
=\{re^{-i\theta}:\infty>r\ge\kappa\}
\cup\{\kappa e^{i\varsigma}:-\theta\le\varsigma\le\theta\}
\cup\{re^{i\theta}:\kappa\le r<\infty\}.
\]
Then, \cite{YZ2026Scarpi} gives \(\kappa_0\ge\max\{1,2c\}\) and \(C>0\) such that
\[
|\Psi_\alpha(z)|\le C|z|^{-\alpha_1},
\qquad
|\Phi_\alpha(z)|\le C|z|^{\alpha_1-1},
\qquad
|z|\ge\kappa_0,\quad |\operatorname{Arg}z|\le\theta.
\]
The kernel construction and the inversion identities in \cite{YZ2026Scarpi} therefore give the following result.

\begin{proposition}
\label{prop:scarpi-kernels}
The inverse Laplace transforms in \eqref{eq:inverse_LT_kernel} are real-valued kernels of exponential order.
For every \(T>0\), they belong to \(L^1(0,T)\), and there is a constant \(C_T>0\) such that
\[
|\psi_\alpha(t)|\le C_Tt^{\alpha_1-1},
\qquad
|\phi_\alpha(t)|\le C_Tt^{-\alpha_1},
\qquad 0<t\le T.
\]
The kernels form a Sonine pair, satisfying \((\psi_\alpha*\phi_\alpha)(t)=1\) for almost every \(t>0\).
Consequently, for any \(u\in W^{1,1}(0,T,X)\), we have
\[{}^{S}I_0^{\alpha(t)}{}^{S}D_0^{\alpha(t)}u=u-u(0) \quad\text{a.e.\ on }(0,T).\]
Also, for any \(f\in L^1(0,T,X)\) such that \(\psi_\alpha*f\in W^{1,1}(0,T,X)\) and \((\psi_\alpha*f)(0)=0\), it holds that
\[
{}^{S}D_0^{\alpha(t)}{}^{S}I_0^{\alpha(t)}f=f \quad\text{a.e.\ on }(0,T).
\]
\end{proposition}

We next establish the full-sector symbol estimates.

\begin{lemma}
\label{lem:scarpi-global-cq-sector}
Let $\alpha_M=\max\{\alpha_1,\alpha_2\}$ and $C_{\log}=\sup_{x>0}\frac{x|\log x|}{1+x^2}$.
Assume
\begin{equation}
\label{eq:scarpi-explicit-admissibility}
\frac{\pi}{2}\alpha_M
+
|\alpha_2-\alpha_1|
\left(
C_{\log}+\frac12|\log c|
\right)
<\pi.
\end{equation}
Then there is an \(\eta\in(0,\pi/2)\) such that
\begin{equation}
\label{eq:scarpi-right-half-plane-angle}
\left|
\operatorname{Arg}\bigl(\Psi_\alpha(z)^{-1}\bigr)
\right|
\le\pi-\eta,
\qquad
\Re z\ge0,\quad z\ne0.
\end{equation}
\end{lemma}

\begin{proof}
Let \(z=re^{i\varphi}\), with \(r>0\) and \(|\varphi|\le\pi/2\), and set \(x=r/c\).
Since \(0\le\Re(c/(c+z))\le1\), we have \(0<\Re B(z)\le\alpha_M\) and
\[
|\Im B(z)|
=|\alpha_2-\alpha_1|
\frac{x|\sin\varphi|}{1+2x\cos\varphi+x^2}
\le |\alpha_2-\alpha_1|\frac{x}{1+x^2}.
\]
Hence
\[
\begin{aligned}
|\Im(B(z)\Log z)|
&\le\alpha_M|\varphi|
+|\alpha_2-\alpha_1|\frac{x|\log(cx)|}{1+x^2}\\
&\le\frac{\pi}{2}\alpha_M
+|\alpha_2-\alpha_1|\left(C_{\log}+\frac12|\log c|\right)
<\pi.
\end{aligned}
\]
Since \(\Psi_\alpha(z)^{-1}=\exp(B(z)\Log z)\), its principal argument is \(\Im(B(z)\Log z)\).
Choosing \(\eta\in(0,\pi/2)\) smaller than the angular margin proves \eqref{eq:scarpi-right-half-plane-angle}.
\end{proof}

The next result extends this bound to a sector wider than the right half-plane and gives the magnitude and derivative estimates.

\begin{corollary}
\label{cor:scarpi-cq-sector-bounds}
Assume \eqref{eq:scarpi-explicit-admissibility}.
Then there exist $\delta_\Sigma,\eta_\Sigma\in(0,\pi/2)$ and a constant $C_\Sigma\ge1$ depending only on \(\alpha_1,\alpha_2,c\), such that, for every \(z\in\Sigma_{\pi/2+\delta_\Sigma}\),
\begin{equation}
\label{eq:scarpi-sector-angle}
\left|
\operatorname{Arg}\bigl(\Psi_\alpha(z)^{-1}\bigr)
\right|
\le\pi-\eta_\Sigma,
\end{equation}
and
\begin{equation}
\label{eq:scarpi-two-regime-bounds}
\begin{aligned}
C_\Sigma^{-1}|z|^{\alpha_1}
&\le
|\Psi_\alpha(z)^{-1}|
\le
C_\Sigma|z|^{\alpha_1},
&& |z|\ge1,
\\
C_\Sigma^{-1}|z|^{\alpha_2}
&\le
|\Psi_\alpha(z)^{-1}|
\le
C_\Sigma|z|^{\alpha_2},
&& 0<|z|\le1.
\end{aligned}
\end{equation}
Moreover,
\begin{equation}
\label{eq:scarpi-log-derivative-bound}
\left|
z\frac{d}{dz}\Psi_\alpha(z)^{-1}
\right|
\le
C_\Sigma|\Psi_\alpha(z)^{-1}|
\qquad
z\in\Sigma_{\pi/2+\delta_\Sigma}.
\end{equation}
\end{corollary}

\begin{proof}
The proof of \cite{YZ2026Scarpi} and the identity \(B(z)-\alpha_2=(\alpha_1-\alpha_2)z/(c+z)\) give
\[
\begin{aligned}
(B(z)-\alpha_1)\Log z
&=O\!\left(\frac{1+\log|z|}{|z|}\right),
&& |z|\to\infty,\\
(B(z)-\alpha_2)\Log z
&=O\!\left(|z|(1+|\log|z||)\right),
&& |z|\to0,
\end{aligned}
\]
uniformly in each fixed closed sector avoiding the negative real axis.
Since \(\Psi_\alpha(z)^{-1}=\exp(B(z)\Log z)\), these estimates give \eqref{eq:scarpi-two-regime-bounds} outside a compact annulus and an angular margin at both ends of a sector slightly wider than the right half-plane.
On the remaining compact annulus, \eqref{eq:scarpi-right-half-plane-angle} and continuity extend the angular margin to a common enlarged sector.
The nonvanishing of \(\Psi_\alpha(z)^{-1}\) extends the magnitude bounds across the annulus.
This proves \eqref{eq:scarpi-sector-angle} and \eqref{eq:scarpi-two-regime-bounds}.
Finally,
\[
z\frac{d}{dz}\Psi_\alpha(z)^{-1} =\Psi_\alpha(z)^{-1} (B(z)+z B'(z) \Log z), \qquad B'(z) = \frac{c(\alpha_1-\alpha_2)}{(c+z)^2}.
\]
Both \(B(z)\) and \(zB'(z)\Log z\) have finite limits at zero and infinity and are bounded on the intervening compact annulus.
This proves \eqref{eq:scarpi-log-derivative-bound}.
\end{proof}

Combining these estimates with Lemma~\ref{lem:fourier-resolvent} yields the operator bounds.

\begin{corollary}
\label{cor:scarpi-operator-transfer-bounds}
Assume \eqref{eq:scarpi-explicit-admissibility} and let $\Rr(z)=(\Psi_\alpha(z)^{-1}\Ii-\Aa)^{-1}$.
For \(z\in\Sigma_{\pi/2+\delta_\Sigma}\), \(\Psi_\alpha(z)^{-1}\in\rho(\Aa)\), and
\begin{equation}
\label{eq:operator-transfer-P-bounds}
\left\|\Rr(z)\right\|_{\mathcal L(X)}
+\left\|z\frac{d}{dz}
\Rr(z)\right\|_{\mathcal L(X)}
\le C
\begin{cases}
|z|^{-\alpha_1},& |z|\ge1,\\
|z|^{-\alpha_2},& 0<|z|\le1,
\end{cases}
\end{equation}
whereas
\begin{equation}
\label{eq:operator-transfer-S-bounds}
\left\|\Psi_\alpha(z)^{-1}\Rr(z)\right\|_{\mathcal L(X)}
+\left\|z\frac{d}{dz}\left[\Psi_\alpha(z)^{-1}\Rr(z)\right]\right\|_{\mathcal L(X)}\le C.
\end{equation}
The constants also hold for Fourier truncations of \(\Aa\), independently of the grid.
\end{corollary}

\begin{proof}
The non-differentiated bounds follow from Lemma~\ref{lem:fourier-resolvent} and Corollary~\ref{cor:scarpi-cq-sector-bounds}.
The identity
\[
\frac{d}{dz}\Rr(z)
=-\left(\frac{d}{dz}\Psi_\alpha(z)^{-1}\right)
\Rr^2(z)
\]
and \eqref{eq:scarpi-log-derivative-bound} give the differentiated bound in \eqref{eq:operator-transfer-P-bounds}.
The product rule gives \eqref{eq:operator-transfer-S-bounds}.
The constants remain unchanged under Fourier truncation by Lemma~\ref{lem:fourier-resolvent}.
\end{proof}

\section{Well-posedness of the Scarpi-type Volterra problem}
\label{sec:well-posedness}
We consider
\begin{equation}
\label{eq:abstract_semi_scarpi}
{}^{S}D_0^{\alpha(t)}u(t)=\Aa u(t)+\Nn(u(t)),
\qquad u(0)=u_0,
\end{equation}
under Assumption~\ref{ass:scarpi-resolvent} and condition~\eqref{eq:scarpi-explicit-admissibility}. Then, from Proposition \ref{prop:scarpi-kernels}, we can express the solution of \eqref{eq:abstract_semi_scarpi} in a Volterra form.

\begin{proposition}
\label{prop:scarpi-volterra-form}
Let \(u\in C([0,T];D(\Aa))\cap W^{1,1}(0,T;X)\), with
\(u(0)=u_0\). Then \(u\) satisfies \eqref{eq:abstract_semi_scarpi}
almost everywhere on \((0,T)\) if and only if it satisfies the Volterra integral equation
\begin{equation}
\label{eq:volterra-nonlinear}
u(t)=u_0+\int_0^t \psi_\alpha(t-\tau)\left[
\Aa u(\tau)+\Nn(u(\tau))\right]\,d\tau.
\end{equation}
\end{proposition}

We say that a solution $u\in C([0,T];D(\Aa))$ is a Volterra solution if it satisfies \eqref{eq:volterra-nonlinear}. 
\subsection{Linear solution families}
Choose \(\theta\in(\pi/2,\pi/2+\delta_\Sigma)\) as in Corollary~\ref{cor:scarpi-cq-sector-bounds} and \(\kappa>0\).
For \(t>0\), define
\begin{align}
\label{eq:P-family}
P(t)&=\frac{1}{2\pi i}\int_{\Gamma_{\theta,\kappa}}
e^{zt}(\Psi_\alpha(z)^{-1}\Ii-\Aa)^{-1}\,dz,\\
\label{eq:S-family}
S(t)&=\frac{1}{2\pi i}\int_{\Gamma_{\theta,\kappa}}
e^{zt}\frac{\Psi_\alpha(z)^{-1}}{z}
(\Psi_\alpha(z)^{-1}\Ii-\Aa)^{-1}\,dz.
\end{align}
The sectorial estimates established in Section~\ref{sec:preliminaries} guarantee that the integrands are analytic in the annular sectors bounded by these contours, so the integral values are independent of \(\kappa\).
By choosing $\kappa$ sufficiently large, the operators $P$ and $S$ coincide with the families introduced in \cite{YZ2026Scarpi}, because their corresponding symbols and the selected branches of the logarithm agree.

\begin{lemma}
\label{lem:S-resolvent-identity}
For every \(T>0\), $P\in C((0,\infty);\mathcal L(X))\cap L^1((0,T);\mathcal L(X))$ and
\begin{equation}
\label{eq:P-local-bound}
\|P(t)\|_{\mathcal L(X)}\le C_Tt^{\alpha_1-1},
\qquad 0<t\le T,
\end{equation}
where $C_T>0$.
The family \(S\) is norm continuous for \(t>0\), locally uniformly bounded, and strongly continuous at zero with \(S(0)=\Ii\).
For \(\Re z\) sufficiently large,
\begin{equation}
\label{eq:PS-Laplace-transforms}
\widehat P(z)=(\Psi_\alpha(z)^{-1}\Ii-\Aa)^{-1},
\qquad
\widehat S(z)=\frac{\Psi_\alpha(z)^{-1}}{z}
(\Psi_\alpha(z)^{-1}\Ii-\Aa)^{-1}.
\end{equation}
For \(x\in D(\Aa)\) and \(t\ge0\),
\[S(t)x\in D(\Aa),\qquad \Aa S(t)x=S(t)\Aa x,\]
and
\begin{align}
\label{eq:S-P-identity}
S(t)x&=x+\int_0^tP(s)\Aa x\,ds,\\
\label{eq:S-volterra-resolvent}
S(t)x&=x+\int_0^t\psi_\alpha(t-s)\Aa S(s)x\,ds.
\end{align}
Moreover, for every \(T>0\), \(P(t)X\subset D(\Aa)\) and \(S'(t)=\Aa P(t)\) in \(\mathcal L(X)\), with
\[
\|\Aa P(t)\|_{\mathcal L(X)}+\|S'(t)\|_{\mathcal L(X)}
\le C_Tt^{-1},\qquad 0<t\le T.
\]
The linear estimates hold uniformly under Fourier truncation.
\end{lemma}

\begin{proof}
The statements through \eqref{eq:S-volterra-resolvent} follow from
\cite{YZ2026Scarpi}.
For the remaining statements, the resolvent identity
\[
\Aa(\Psi_\alpha(z)^{-1}\Ii-\Aa)^{-1}
=\Psi_\alpha(z)^{-1}(\Psi_\alpha(z)^{-1}\Ii-\Aa)^{-1}-\Ii
\]
and \eqref{eq:operator-transfer-S-bounds} show that the integral in \eqref{eq:P-family}, applied to any \(x\in X\), converges in
\(D(\Aa)\) with its graph norm. Closedness of \(\Aa\) and
\(\int_{\Gamma_{\theta,\kappa}}e^{zt}\,dz=0\) give
\[
\Aa P(t)=\frac{1}{2\pi i}\int_{\Gamma_{\theta,\kappa}}
e^{zt}\Psi_\alpha(z)^{-1}(\Psi_\alpha(z)^{-1}\Ii-\Aa)^{-1}\,dz.
\]
Differentiation of \eqref{eq:S-family} on compact time intervals in \((0,\infty)\) gives \(S'(t)=\Aa P(t)\).
On \(\Gamma_{\theta,t^{-1}}\), the uniformly bounded operator
factor gives ray and arc estimates of order \(t^{-1}\).
All these constants are uniform under Fourier truncation by Lemma~\ref{lem:fourier-resolvent}.
\end{proof}

For \(u_0\in X\) and \(f\in C([0,T];X)\), the linear mild solution is
\[
u(t)=S(t)u_0+\int_0^tP(t-s)f(s)\,ds.
\]
This formula defines a continuous \(X\)-valued function. For
\(u\in C([0,T];D(\Aa))\), it is equivalent to
\(u=u_0+\psi_\alpha*(\Aa u+f)\).
Indeed, \eqref{eq:PS-Laplace-transforms} gives this equivalence for each Fourier coefficient by the scalar Volterra resolvent identities and uniqueness of the scalar Volterra equation.
Both sides of the vector-valued Volterra equation belong to
\(C([0,T];X)\), so equality of their Fourier coefficients proves
the identity in \(X\). This also proves the converse implication.

\subsection{Well-posedness and Temporal regularity}
A mild solution of \eqref{eq:abstract_semi_scarpi} is a function
\(u\in C([0,T];X)\) satisfying
\begin{equation}
\label{eq:semilinear-mild-solution}
u(t)=S(t)u_0+\int_0^tP(t-s)\Nn(u(s))\,ds.
\end{equation}

\begin{theorem}
\label{thm:maximal-blowup}
For every \(u_0\in X\), there is a unique maximal mild solution
\[
u\in C([0,T_{\max});X),\qquad 0<T_{\max}\le\infty.
\]
Two mild solutions with the same initial value agree on every common interval of existence. Every mild solution on a closed interval
\([0,T_0]\) extends uniquely beyond \(T_0\), with the convolution in
\eqref{eq:semilinear-mild-solution} still starting at zero.
If \(T_{\max}<\infty\), then
\[
\limsup_{t\uparrow T_{\max}}\|u(t)\|_X=\infty.
\]
\end{theorem}

\begin{proof}
Lemma~\ref{lem:S-resolvent-identity} gives the strong continuity and local boundedness of \(S\), together with the local integrability of
\(P\). 
Assumption~\ref{ass:scarpi-resolvent} supplies the Lipschitz condition on bounded subsets of \(X\). The conclusions therefore follow from \cite{YZ2026Scarpi}.
\end{proof}

After establishing the existence of a maximal mild solution, we next give the regularity of the solution and its derivative, which is needed for CQ error analysis.

\begin{theorem}
\label{thm:semilinear-regularity}
Assume that \(\Nn\) is continuously Fr\'echet differentiable and that \(\Nn'\) is Lipschitz on bounded subsets of \(X\).
Let \(u_0\in D(\Aa)\), and let \(u\) be the maximal mild solution
from Theorem~\ref{thm:maximal-blowup}. For every \(T<T_{\max}\),
\[
u\in C([0,T];D(\Aa))\cap W^{1,1}(0,T;X),
\]
its derivative is continuous on \((0,T]\), and
\begin{equation}
\label{eq:semilinear-derivative-bound}
\|u'(t)\|_X\le C_Tt^{\alpha_1-1},\qquad 0<t\le T.
\end{equation}
In particular, \(u\) is H\"older continuous of order \(\alpha_1\) in \(X\), satisfies \eqref{eq:volterra-nonlinear}, and is a strong solution of \eqref{eq:abstract_semi_scarpi} on \((0,T)\).
\end{theorem}

This result can be directly extended to the case ${}^{S}D_0^{\alpha(t)}u=\Aa u+\Nn(u)+f(t)$. 
The mild formula contains the additional term \(P*f\).  For \(f\in C([0,\infty);X)\), local and maximal well-posedness
follow from the same contraction and full-history continuation, with the continuous prescribed term \(S(\cdot)u_0+P*f\).
If, in addition, $f\in W^{1,1}(0,T;X)$ and $\|f'(t)\|_X\le Ct^{\alpha_1-1}$ almost everywhere,
then \eqref{eq:P-local-bound} yields
$\|(P*f')(t)\|_X\le C_Tt^{\alpha_1-1}$.
The derivative reconstruction and \eqref{eq:graph-cancellation} therefore give the same temporal regularity for $T<T_{\max}$.

\section{BE-CQ time stepping and Fourier pseudospectral solver}
\label{sec:be-cq}
We apply backward Euler convolution quadrature (BE--CQ) to
\eqref{eq:volterra-nonlinear}, using the Laplace symbol
\(\Psi_\alpha\) directly
\cite{Lubich1988I,Lubich1988II,GarrappaGiusti2023}.
The continuous formulation and solution families were established in the preceding section using \cite{YZ2026Scarpi}.

\subsection{BE-CQ scheme}

Let $t_n=nh,n=0,1,\dots,N$ with $T=Nh$. We denote the generating polynomial of the backward differentiation formulas (BDFs) by $\delta(\zeta)$. For the corresponding backward Euler (BE) symbol \(\delta(\zeta)=1-\zeta\), define
\[
\Psi_h(\zeta)
:=\Psi_\alpha\!\left(\frac{1-\zeta}{h}\right)
=\left(\frac{1-\zeta}{h}\right)^{-B((1-\zeta)/h)}
=\sum_{n=0}^{\infty}\omega_n\zeta^n,
\qquad |\zeta|<1,
\]
using the principal logarithm fixed in Section~\ref{sec:preliminaries}.
Since \(\Re((1-\zeta)/h)>0\), the symbol and operator estimates in
Corollaries~\ref{cor:scarpi-cq-sector-bounds} and
\ref{cor:scarpi-operator-transfer-bounds} apply under
\eqref{eq:scarpi-explicit-admissibility}, with constants independent
of \(h\) and the Fourier grid. In particular,
\[
\omega_0=\Psi_\alpha(1/h)=h^{B(1/h)}>0,
\qquad
\omega_0=h^{\alpha_1}\bigl(1+O(h|\log h|)\bigr).
\]

Let $u^0=u_0$, then for \(1\le n\le N\), the BE-CQ discretization of \eqref{eq:volterra-nonlinear} reads
\begin{equation}
\label{eq:pde-cq-volterra}
u^n
=
u_0+\sum_{j=0}^{n}\omega_{n-j}
\bigl(\Aa u^j+\Nn(u^j)\bigr).
\end{equation}
Thus, when advancing from \(t_{n-1}\) to \(t_{n}\), the known history part is
\begin{equation}
\label{eq:pde-history}
H_u^n
:=
u_0+\sum_{j=0}^{n-1}\omega_{n-j}
\bigl(\Aa u^j+\Nn(u^j)\bigr),
\end{equation}
and the current step satisfies
\(
u^{n}
=
H_u^n+\omega_0\bigl(\Aa u^{n}+\Nn(u^{n})\bigr).
\)

\subsection{Cauchy--FFT computation of the BE-CQ weights}
\label{subsec:cauchy-fft}
Since \(\omega_n\) are Taylor coefficients of \(\Psi_h\), Cauchy's formula
on a circle \(|\zeta|=\rho<1\) gives
\[
\omega_n
=
\frac{1}{2\pi i}
\int_{|\zeta|=\rho}
\frac{\Psi_h(\zeta)}{\zeta^{n+1}}\,d\zeta
=
\frac{\rho^{-n}}{2\pi}
\int_0^{2\pi}
\Psi_h(\rho e^{i\theta})e^{-in\theta}\,d\theta .
\]
For an integer \(J_{\rm FFT}>N\), applying the \(J_{\rm FFT}\)-point trapezoidal rule  gives
\begin{equation}
\label{eq:pde-fft-weights}
\check\omega_n
=
\frac{\rho^{-n}}{J_{\rm FFT}}
\sum_{\ell=0}^{J_{\rm FFT}-1}
\Psi_h\!\left(\rho e^{2\pi i\ell/J_{\rm FFT}}\right)
e^{-2\pi i n\ell/J_{\rm FFT}},
\qquad 0\le n\le N.
\end{equation}
All weights $\{\check\omega_n\}_{n=0}^N$ are thus computed via FFT and are real in exact arithmetic by conjugate symmetry. 

This approach essentially adapts the coefficient-recovery method developed in \cite{GarrappaGiusti2023} for exponential-type Scarpi kernels, with the BE symbol replacing the continuous Laplace variable in Lubich's CQ construction. To compute the resulting BE--CQ weights, we directly employ the Cauchy--FFT procedure and the associated error estimates established by Garrappa and Giusti.

\begin{proposition}\label{prop:pde-cauchy-fft}
Let \(0<\rho<\rho_r<1\), and assume that \(\Psi_h\) is analytic in
\(|\zeta|\le \rho_r\). Then for \(0\le n\le J_{\rm FFT}-1\),
\[
\left|\omega_n-\check\omega_n\right|
\le
\frac{C_{\Psi,r}\rho^{-n}}{(\rho_r/\rho)^{J_{\rm FFT}}-1},
\qquad
C_{\Psi,r}:=\max_{|\zeta|=\rho_r}|\Psi_h(\zeta)|.
\]
\end{proposition}

\paragraph{Choice of $\rho$ and $J_{\rm FFT}$}
Given a tolerance \(\xi\), Proposition \ref{prop:pde-cauchy-fft} implies that,
with \(J_{\rm FFT}>N\), it is sufficient to choose
\[
J_{\rm FFT}
\ge
\frac{\log\!\left(C_{\Psi,r}\rho^{-N}+\xi\right)-\log\xi}
{\log \rho_r-\log\rho}
\]
to keep the contour quadrature error below \(\xi\) for \(0\le n\le N\). In
practice \(C_{\Psi,r}\) is estimated by sampling \(\Psi_h\) on \(|\zeta|=\rho_r\).
The radius \(\rho\) is chosen to balance coefficient amplification against floating
point round-off. Suppose that \(\varepsilon_{\rm mach}\) is machine precision and
\(0<F_s<1\) is a safety
factor, a common choice is
\[
\rho^{-N}C_{\Psi,r}\varepsilon_{\rm mach}=F_s\xi,
\qquad
\rho
=
\exp\!\left[
-\frac1N
\log\!\left(
\frac{F_s\xi}{C_{\Psi,r}\varepsilon_{\rm mach}}
\right)
\right],
\]
provided \(0<\rho<\rho_r<1\). For the numerical settings used below
and \(N\ge1\), if the balance ratio
\(F_s\xi/(C_{\Psi,r}\varepsilon_{\rm mach})\) is nonfinite or at most one,
the implementation first sets \(\rho=e^{-0.5/N}\).
It then clips \(\rho\) to \([10^{-14},0.999999]\).
If \(\rho\ge\rho_r\), it sets
\(\rho_r=\min\{(1+\rho)/2,0.999999\}\), reduces \(\rho\) to
\(0.8\rho_r\) if still necessary, and recomputes the sampled bound. The final weights are checked by recomputing them with
\(2J_{\rm FFT}\) quadrature points and requiring
\[
\frac{\max_{0\le n\le N}
 |\check\omega_n-\check\omega_n^{(2J_{\rm FFT})}|}
 {\max_{0\le n\le N}|\check\omega_n^{(2J_{\rm FFT})}|}
\le \xi_{\rm check}.
\]
Here \(\xi_{\rm check}>0\) is the prescribed a posteriori verification tolerance.
The exact Taylor coefficient remains denoted by \(\omega_n\), while \(\check\omega_n\) denotes the approximation produced by the \(J_{\rm FFT}\)-point Cauchy quadrature.

\subsection{Discrete Scarpi resolvent stability and temporal consistency}
\label{subsec:discrete-scarpi-stability}

Using \eqref{eq:PS-Laplace-transforms}, define the operator coefficients
\(
P_h^j,S_h^j\in\mathcal L(X)
\)
through
\begin{equation}
    \label{eq:discrete-PS-family}
    \begin{aligned}
\sum_{j=0}^{\infty}P_h^j\zeta^j
=\left[
\Psi_h(\zeta)^{-1}\Ii-\Aa
\right]^{-1},
\quad
\sum_{j=0}^{\infty}S_h^j\zeta^j
=\Psi_h(\zeta)^{-1}
\left[
\Psi_h(\zeta)^{-1}\Ii-\Aa
\right]^{-1}.
\end{aligned}
\end{equation}
These are exactly the BE substitutions \(z=(1-\zeta)/h\) in the continuous
transfer functions \(\widehat P(z)\) and \(z\widehat S(z)\), respectively.

\begin{lemma}
\label{lem:discrete-scarpi-resolvent}
Assume
condition~\eqref{eq:scarpi-explicit-admissibility}.  For fixed \(h_0,T>0\),
\(0<h\le h_0\), and \(t_j\le T\),
\begin{equation}
\label{eq:discrete-P-S-bound}
\|P_h^j\|_{\mathcal L(X)}
\le Ch\left(t_{j+1}^{\alpha_1-1}+t_{j+1}^{\alpha_2-1}\right),
\qquad
\|S_h^j\|_{\mathcal L(X)}
\le\frac{C}{j+1}.
\end{equation}
The constants are independent of \(h,j\) and the Fourier grid.
Consequently, if
\(t_n\le T\),
\begin{equation}
\label{eq:discrete-resolvent-sums}
\sum_{j=0}^{n}\|P_h^j\|_{\mathcal L(X)}\le C_T,
\qquad
\sum_{j=0}^{n}\|S_h^j\|_{\mathcal L(X)}
\le C\bigl(1+\log(n+1)\bigr).
\end{equation}
\end{lemma}

\begin{proof}
Corollary~\ref{cor:scarpi-operator-transfer-bounds} gives the
sectorial bounds and derivative estimates for \(\widehat P(z)\) and
\(z\widehat S(z)\), uniformly under Fourier truncation.
The operator-valued BE--CQ coefficient estimate
\cite{Lubich1988I,Lubich1988II,CuestaLubichPalencia2006}
gives the bound for \(P_h^j\).
For the bounded symbol \(z\widehat S(z)\), deforming the Cauchy
coefficient contour around \(\zeta=1\) with radius \((j+1)^{-1}\)
gives \(C/(j+1)\) for \(j\ge1\), and \(j=0\) follows directly
from the symbol bound. These estimates give
\eqref{eq:discrete-P-S-bound}. Summation proves
\eqref{eq:discrete-resolvent-sums}.
\end{proof}

\begin{theorem}
\label{thm:fully-discrete-stability}
Let \(U^n,V^n\in D(\Aa)\) and \(r_{\rm ext}^n\in X\) for \(0\le n\le N\), with
\(U^0=V^0\) and \(r_{\rm ext}^0=0\), and suppose
\begin{align*}
U^n&=U^0+\sum_{j=0}^{n}\omega_{n-j}
\bigl(\Aa U^j+\Nn(U^j)\bigr),\\
V^n&=V^0+\sum_{j=0}^{n}\omega_{n-j}
\bigl(\Aa V^j+\Nn(V^j)\bigr)+r_{\rm ext}^n,
\qquad 1\le n\le N.
\end{align*}
Assume condition~\eqref{eq:scarpi-explicit-admissibility} and suppose that both $\{U^n\}$ and $\{V^n\}$ remain in a fixed bounded set, independent of $h$, on which \(\Nn:X\to X\) is Lipschitz with constant $L_\Nn$.  Then, for all sufficiently small \(h\)
\[
\max_{0\le n\le N}\|V^n-U^n\|
\le C_T\bigl(1+\log(N+1)\bigr)
\max_{0\le n\le N}\|r_{\rm ext}^n\|,
\]
where \(C_T\) is independent of \(h\).
\end{theorem}

\begin{proof}
Subtracting the recursions and using
\eqref{eq:discrete-PS-family} gives
\[
V^n-U^n
=\sum_{j=1}^{n}P_h^{n-j}
\bigl(\Nn(V^j)-\Nn(U^j)\bigr)
+\sum_{j=1}^{n}S_h^{n-j}r_{\rm ext}^j.
\]
Lemma~\ref{lem:discrete-scarpi-resolvent} yields
\[
\max_{n\le N}\left\|\sum_{j=1}^{n}S_h^{n-j}r_{\rm ext}^j\right\|
\le C(1+\log(N+1))\max_{j\le N}\|r_{\rm ext}^j\|.
\]
Since \(\|P_h^0\|\le C(h^{\alpha_1}+h^{\alpha_2})\), choose \(h\)
so that \(L_\Nn\|P_h^0\|\le1/2\). Absorbing the current term gives
\[
\|V^n-U^n\|
\le 2L_\Nn\sum_{j=1}^{n-1}\|P_h^{n-j}\|\,\|V^j-U^j\|
+2\left\|\sum_{j=1}^{n}S_h^{n-j}r_{\rm ext}^j\right\|.
\]
For \(h\le\ell_T\), the positive-lag sum on an interval of length
\(\ell_T\) is bounded by
\(C(\ell_T^{\alpha_1}+\ell_T^{\alpha_2})\).
Choose \(\ell_T>0\), independently of \(h\), so that
\(2L_\Nn C(\ell_T^{\alpha_1}+\ell_T^{\alpha_2})\le1/2\).
On each interval, absorb its recent history and use the already
controlled preceding history. A finite induction over at most
\(\lceil T/\ell_T\rceil+1\) intervals, together with
\eqref{eq:discrete-resolvent-sums}, proves the claim.
\end{proof}

\paragraph{Temporal consistency}

Theorem~\ref{thm:semilinear-regularity} supplies the regularity
assumptions below on every \([0,T]\) with \(T<T_{\max}\), under its hypotheses. 
The BE--CQ scheme is first-order accurate away from the initial time, as proved below.

The high-frequency expansion
\(\Psi_\alpha(z)=z^{-\alpha_1}
[1+O((1+\log|z|)/|z|)]\), followed by contour inversion, gives
\(\psi_\alpha(t)\sim t^{\alpha_1-1}/\Gamma(\alpha_1)\) as
\(t\downarrow0\). For a Volterra solution continuous in \(D(\Aa)\),
\[
u(t)=u_0+
\frac{t^{\alpha_1}}{\Gamma(1+\alpha_1)}
\bigl(\Aa u_0+\Nn(u_0)\bigr)+o(t^{\alpha_1}).
\]
The derivative bound in Theorem~\ref{thm:semilinear-regularity}
is of order \(t^{\alpha_1-1}\). The error estimate below is
\(O(h^{\alpha_1})\) at \(t_1=h\) and \(O(h)\) for
\(t_n\ge t_*>0\).

\begin{theorem}
\label{thm:temporal-consistency}
Assume condition~\eqref{eq:scarpi-explicit-admissibility}. Let
\(u\) be a Volterra solution of \eqref{eq:volterra-nonlinear} such that
\[
u\in C([0,T];D(\Aa))\cap W^{1,1}(0,T;X)
\]
and
\[
\|u'(t)\|_X
\le C t^{\alpha_1-1},
\qquad
0<t\le T.
\]
Let
\(\{u^n\}_{n=0}^N\subset D(\Aa)\) be a solution of the
full-history recursion~\eqref{eq:pde-cq-volterra}, with
\(u^0=u_0\).
Suppose that \(\Nn:X\to X\) is Lipschitz on a fixed bounded set containing the exact and discrete solutions, independently of \(h\).  Then, for all sufficiently small \(h\) and every fixed
\(t_*>0\),
\[
\max_{t_*\le t_n\le T}
\|u^n-u(t_n)\|_X
\le C_* h,
\]
where \(C_*\) is independent of \(h\).
\end{theorem}

\begin{proof}
The linear mild--Volterra equivalence in the section 3 and
\eqref{eq:S-P-identity} give
\[
u(t)-u_0
=
\int^t_0 P(t-s) g(s) ds,
\qquad
g(t):=\Aa u_0+\Nn(u(t)).
\]

For \(1\le n\le N\), the solution \(u\) satisfies
\begin{equation}
\label{eq:exact-samples-perturbed-cq}
u(t_n)=u_0
+
\sum_{j=0}^n\omega_{n-j}
\bigl[\Aa u(t_j)+\Nn(u(t_j))\bigr] + r_{\rm cons}^n,
\end{equation}
with \(r_{\rm cons}^0:=0\). Thus, we define the BE-CQ consistency error as
\[
\varepsilon_{\rm cons}^n
:=
-\sum_{j=1}^nS_h^{n-j}r_{\rm cons}^j .
\]
The generating functions in \eqref{eq:discrete-PS-family} satisfy
\(\sum_{j\ge0}P_h^j\zeta^j
=\Psi_h(\zeta)\sum_{j\ge0}S_h^j\zeta^j\).
Because the time-step recursion does not apply at $n=0$, explicitly enforcing the initial condition yields
\begin{equation}
\label{eq:resolved-consistency-identity}
\varepsilon_{\rm cons}^n
=
\sum_{j=0}^nP_h^{\,n-j}g(t_j)
-
\int_0^{t_n}P(t_n-s)g(s)\,ds
-
\omega_0S_h^n g(0).
\end{equation}

On a fixed right-shifted subsector of
\(\Sigma_{\pi/2+\delta_\Sigma}\),
Corollary~\ref{cor:scarpi-operator-transfer-bounds} gives
\(\|\widehat P(z)\|\le C|z|^{-\alpha_1}\).
The BE--CQ estimate \cite{Lubich1988I}, applied to
the constant function \(g(0)\), gives
\begin{equation}
\label{eq:CQ-error-constant}
\left\|
\sum_{j=0}^{n}P_h^j g(0)
-
\int_0^{t_n}P(s) g(0)\,ds
\right\|_X
\le
C h t_n^{\alpha_1-1}\|g(0)\|_X,
\qquad 1\le n\le N .
\end{equation}
The proof of that theorem gives the same constant-input estimate
at non-grid times \(h\le t\le T\), with the sum over \(jh\le t\).
For \(0<t<h\), the corresponding operator error is bounded by
\(\|P_h^0\|+\int_0^t\|P(s)\|\,ds
\le C(h^{\alpha_1}+t^{\alpha_1})
\le Ch t^{\alpha_1-1}\).

Since \(X\) is Hilbert, absolute continuity of \(u\) and
Lipschitz continuity of \(\Nn\) on its range imply
\(g\in W^{1,1}(0,T;X)\) and
\(\|g'(t)\|_X\le Ct^{\alpha_1-1}\).
Use \(g(t)-g(0)=(1*g')(t)\) and the convolution associativity identities
\cite{Lubich1988I}.
The constant-input estimate, including non-grid times, then gives
\begin{align}
&\left\|
\sum_{j=0}^nP_h^{\,n-j}[g(t_j)-g(0)]
-
\int_0^{t_n}P(t_n-s)[g(s)-g(0)]\,ds
\right\|_X
\nonumber\\
&\qquad\le
Ch
\int_0^{t_n}
(t_n-s)^{\alpha_1-1}
s^{\alpha_1-1}\,ds
=
C h \frac{\Gamma(\alpha_1)^2}{\Gamma(2\alpha_1)}t_n^{2\alpha_1-1},
\label{eq:CQ-error-remainder}
\end{align}

The remaining starting term in
\eqref{eq:resolved-consistency-identity} is also of the required order.
Indeed, by
\(\omega_0=O(h^{\alpha_1})\) and
Lemma~\ref{lem:discrete-scarpi-resolvent},
\[
\|\omega_0S_h^ng(0)\|_X
\le
C\frac{h^{\alpha_1}}{n+1}\|g(0)\|_X
\le
C h t_n^{\alpha_1-1}\|g(0)\|_X,
\qquad n\ge1,
\]
where we used \(h\le t_n\).
Combining this estimate with
\eqref{eq:CQ-error-constant} and
\eqref{eq:CQ-error-remainder} yields
\[
\|\varepsilon_{\rm cons}^n\|_X
\le
C h
\left(
t_n^{\alpha_1-1}\|g(0)\|_X
+
t_n^{2\alpha_1-1}
\right)
\le
C_T h t_n^{\alpha_1-1},
\qquad
1\le n\le N .
\]

Comparing \eqref{eq:exact-samples-perturbed-cq} with \eqref{eq:pde-cq-volterra} and using \eqref{eq:discrete-PS-family} gives
\[ u^n-u(t_n) = \sum_{j=1}^nP_h^{\,n-j}
\bigl[\Nn(u^j)-\Nn(u(t_j))\bigr]
+\varepsilon_{\rm cons}^n .\]
Using the Lipschitz continuity of $\Nn$ and Lemma~\ref{lem:discrete-scarpi-resolvent},
after absorption of the current term for small \(h\), we derive 
\[
\|u^n-u(t_n)\|_X
\le C L_\Nn h\sum_{j=1}^{n-1}
\left(t_{n-j+1}^{\alpha_1-1}+t_{n-j+1}^{\alpha_2-1}\right)
\|u^j-u(t_j)\|_X + C_T h t_n^{\alpha_1-1}.
\]
The discrete convolution estimate
\[
h\sum_{j=1}^{n-1}
t_{n-j+1}^{\alpha_i-1}t_j^{\alpha_1-1}
\le \int_0^{t_n}
(t_n-s)^{\alpha_i-1}s^{\alpha_1-1}\,ds
=\frac{\Gamma(\alpha_i)\Gamma(\alpha_1)}
{\Gamma(\alpha_i+\alpha_1)}
t_n^{\alpha_i+\alpha_1-1}
\]
for \(i\in\{1,2\}\) shows that the convolution preserves the weight
\(t_n^{\alpha_1-1}\) on finite time intervals. 
By choosing $\ell>0$ (independent of $h$) such that $C L_\Nn(\ell^{\alpha_1}+\ell^{\alpha_2})<\frac12$, and applying the finite-block induction from Theorem~\ref{thm:fully-discrete-stability} sequentially across interval blocks starting with $[0,\ell]$, we obtain
\[\|u^n-u(t_n)\|_X
\le
C_T h\,t_n^{\alpha_1-1},
\qquad
1\le n\le N.\]
Restricting to \(t_n\ge t_*>0\) proves the stated estimate.
\end{proof}
\begin{remark}
   The same proof applies to the case $g(t)=\Aa u_0+\Nn(u(t))+f(t)$ under the assumptions of Section~3. The stability theorem remains valid when both sequences employ identical forcing terms, as these cancel exactly in the difference equation.
\end{remark}

\subsection{Fourier discretization and Newton--GMRES solver}\label{subsec:fourier_newton_solver}
Let \(X_{\boldsymbol M}\) denote the periodic grid space on a uniform
Fourier grid with \(M_{\rm grid}=\prod_{i=1}^dM_i\) nodes.  We use the normalized discrete inner product
\[
\langle U,V\rangle_{\boldsymbol M}
=
\frac{|\Omega|}{M_{\rm grid}}
\sum_{\boldsymbol j}U_{\boldsymbol j}\overline{V_{\boldsymbol j}},
\qquad
\|U\|_{\boldsymbol M}
=
\langle U,U\rangle_{\boldsymbol M}^{1/2}.
\]
The Fourier pseudospectral realization \(\Aa_M:X_{\boldsymbol M}\to X_{\boldsymbol M}\) of
\(\Aa=-d_u(-\Delta)^{\beta/2}\) is defined by
\[
[\FT_h(\Aa_M U)]_{\boldsymbol k}
=
-d_u|\boldsymbol\kappa_{\boldsymbol k}|^\beta
(\FT_hU)_{\boldsymbol k}.
\]
Define the nodal reaction by \((\Nn_M(U))_{\boldsymbol j}=\Nn(U_{\boldsymbol j})\).
The Fourier resolvent estimates also hold for \(\Aa_M\) in
\(\|\cdot\|_{\boldsymbol M}\) with grid-independent constants.
Thus the preceding time-discrete estimates apply on \(X_{\boldsymbol M}\)
when the relevant Lipschitz constants are uniform on the grids.
Let \(H^n\) be the grid version of \eqref{eq:pde-history}, with
\(u_0,u^j\) replaced by \(U^0,U^j\).  At each time step, solve
\begin{equation}
\label{eq:fully-discrete-implicit}
G_n(U^n)=0,
\qquad
G_n(V):=V-H^n-\omega_0\bigl(\Aa_M V+\Nn_M(V)\bigr).
\end{equation}

On the invariant real grid subspace, assume that \(\Nn_M\) satisfies the mesh-independent one-sided Lipschitz condition
\begin{equation}
\label{eq:one-sided-reaction}
\Re\langle\Nn_M(V)-\Nn_M(W),V-W\rangle_{\boldsymbol M}
\le
\lambda_{\Nn}\|V-W\|_{\boldsymbol M}^2.
\end{equation}
Since \(-\Aa_M\) is nonnegative and self-adjoint,
\[
\Re\langle G_n(V)-G_n(W),V-W\rangle_{\boldsymbol M}
\ge (1-\omega_0\lambda_{\Nn})\|V-W\|_{\boldsymbol M}^2.
\]
If \(\omega_0\lambda_{\Nn}<1\), the continuous map \(G_n\) is
strongly monotone and coercive, so the finite-dimensional current
equation has a unique solution. This condition holds uniformly on the
grids for small \(h\), and for every \(h>0\) if
\(\lambda_{\Nn}\le0\).

Starting from \(U^{n,(0)}=U^{n-1}\), an inexact Newton correction
\(s_k\) satisfies
\begin{equation}
\label{eq:inexact-newton-equation}
\begin{aligned}
\bigl[I-\omega_0\Aa_M-\omega_0\Nn_M'(U^{n,(k)})\bigr]s_k
&=-G_n(U^{n,(k)})+\mu_k^{\rm lin},\\
\|\mu_k^{\rm lin}\|_{\boldsymbol M}&\le\eta_k\|G_n(U^{n,(k)})\|_{\boldsymbol M},
\end{aligned}
\end{equation}
where \(0\le\eta_k\le\eta_{\max}<1\), with \(\eta_{\max}\) fixed.  The update is
\[
U^{n,(k+1)}=U^{n,(k)}+\lambda_ks_k,
\qquad 0<\lambda_k\le1.
\]

Use Armijo backtracking with a parameter in \((0,1/2)\), and define
\[
\phi_k(\lambda)
:=
\frac12
\left\|
G_n\bigl(U^{n,(k)}+\lambda s_k\bigr)
\right\|_{\boldsymbol M}^2.
\]
Then
\(
\phi_k'(0)\le
-(1-\eta_{\max})\|G_n(U^{n,(k)})\|_{\boldsymbol M}^2,
\)
so the correction is a descent direction whenever \(G_n(U^{n,(k)})\ne0\).
Moreover, if $1-\omega_0\lambda_{\Nn}>0$, then for every grid vector \(Z\), differentiating \eqref{eq:one-sided-reaction} in the direction \(Z\) implies
\[
\Re \langle G_n'(V)Z,Z\rangle_{\boldsymbol M}
\ge(1-\omega_0\lambda_{\Nn})\|Z\|_{\boldsymbol M}^2,
\qquad
\|G_n'(V)^{-1}\|\le(1-\omega_0\lambda_{\Nn})^{-1}.
\]

If \(\Nn_M'\) is locally Lipschitz, strong monotonicity also makes every residual level set bounded.  
Standard damped inexact Newton theory therefore gives convergence to the unique current solution and superlinear local convergence when \(\eta_k\to0\) \cite{KnollKeyes2004}.

\begin{example}
    For the Nagumo reaction
\(\Nn(u)=u(1-u)(u-a)\), \(0<a<1\),
\[
\Nn'(u)=-3u^2+2(1+a)u-a
\le \frac{1-a+a^2}{3}.
\]
Thus \eqref{eq:one-sided-reaction} holds globally with
\(\lambda_{\Nn}=(1-a+a^2)/3\).
\end{example}
\begin{example}
    For the FitzHugh--Nagumo case
\(
\Nn(u,v)
=\begin{pmatrix}u(1-u)(u-a)-v\\
r_v(b u-\gamma v)
\end{pmatrix},
\)
with \(0<a<1\) and \(b,r_v,\gamma>0\), the symmetric part of its Jacobian satisfies
\[
\frac{\Nn'(u,v)+\Nn'(u,v)^\top}{2}
\preceq
\begin{pmatrix}
(1-a+a^2)/3 &(r_vb-1)/2\\
(r_vb-1)/2 &-r_v\gamma
\end{pmatrix}.
\]
The largest eigenvalue of this matrix is therefore an admissible one-sided Lipschitz constant $\lambda_{\Nn}$. 
The Lipschitz constants in the preceding stability estimates are mesh-independent whenever the grid values remain in a common bounded pointwise range.
\end{example}

\paragraph{Reaction-dependent preconditioning}
Write \(J_k=G_n'(U^{n,(k)})\), and let \(\overline J_k\) be a spatially constant scalar, diagonal, or small component-block
approximation of the discrete reaction Jacobian \(\Nn_M'(U^{n,(k)})\). Define
\[
\mathcal Q_h=I-\omega_0\Aa_M,
\qquad
\mathcal P_{h,k}=\mathcal Q_h-\omega_0\overline J_k.
\]
The choice \(\overline J_k=0\) gives diffusion-only preconditioning. For uniformly nonsingular Fourier-mode matrices, 
applying \(\mathcal P_{h,k}^{-1}\) requires FFTs and one scalar division or small block solve per mode, at cost
\(\mathcal O(M_{\rm grid}\log M_{\rm grid})\).

Right-preconditioned GMRES \cite{SaadSchultz} computes
\(s_k=\mathcal P_{h,k}^{-1}y_k\) until the true linear residual
\(J_ks_k+G_n(U^{n,(k)})\) satisfies
\eqref{eq:inexact-newton-equation}.
Since
\(\|v\|_{\boldsymbol M}
=(|\Omega|/M_{\rm grid})^{1/2}\|v\|_2\),
relative Euclidean and relative \(\boldsymbol M\)-norm residuals coincide. Neither the Newton matrix nor the reaction Jacobian is assembled.

For \(s=\mathcal P_{h,k}^{-1}y\), the splitting
\(J_k=\mathcal P_{h,k}
-\omega_0(\Nn_M'(U^{n,(k)})-\overline J_k)\)
gives
\begin{equation}
\label{eq:reaction-aware-reused-action}
J_k\mathcal P_{h,k}^{-1}y
=
y-\omega_0
\bigl(\Nn_M'(U^{n,(k)})-\overline J_k\bigr)s.
\end{equation}
Each Krylov product therefore uses one preconditioner application and a pointwise reaction correction, without a separate application of \(\Aa_M\).
For \(\overline J_k=0\), the bound
\(\|\mathcal Q_h^{-1}\|\le1\) shows that the preconditioned operator is a mesh-uniform \(\mathcal O(\omega_0)\)
perturbation of the identity if \(\Nn_M'\) is uniformly bounded on the numerical range.

The identity \(\mathcal Q_hs=y+\omega_0\overline J_ks\) also gives
\begin{equation}
\label{eq:reaction-aware-armijo-reuse}
\begin{aligned}
G_n(U^{n,(k)}+\lambda s)
={}&G_n(U^{n,(k)})
+\lambda\bigl(y+\omega_0\overline J_ks\bigr)\\
&-\omega_0\bigl[
\Nn_M(U^{n,(k)}+\lambda s)-\Nn_M(U^{n,(k)})
\bigr].
\end{aligned}
\end{equation}
Thus Armijo trials require only pointwise reaction evaluations and vector operations. 
The nonlinear stopping test is verified with the true residual.

For scalar reactions, set
\(\overline J_k=(d_{\min}+d_{\max})/2\), where
\(d_{\min}\) and \(d_{\max}\) are the extrema of the diagonal entries of
\(\Nn_M'(U^{n,(k)})\).
For systems such as FitzHugh--Nagumo, \(\overline J_k\) retains the constant couplings and second-row derivatives.
If a shifted Fourier-mode matrix fails the nonsingularity check, use \(\overline J_k=0\).

\begin{remark}[Spatial discretization]
Let \(I_{\boldsymbol M}\) denote the periodic trigonometric interpolation operator and set \(M=\min_i M_i\). 
We implicitly identify the grid vectors with their trigonometric interpolants. When applying the discrete operators \(\Aa_M\) or \(\Nn_M\), the term \(I_{\boldsymbol M}v\) simply represents the corresponding nodal values. 
For any \(v\in H_{\rm per}^s(\Omega)\) with \(s>d/2+\beta\), the standard Fourier approximation estimates \cite{ShenTangWang2011} provide
\[
\|v-I_{\boldsymbol M}v\|_{L^2}
+
M^{-\beta}
\|I_{\boldsymbol M}\Aa v-\Aa_M I_{\boldsymbol M}v\|_{\boldsymbol M}
\le
CM^{-s}\|v\|_{H^s}.
\]
We assume that the exact solution maintains this \(H^s\)-regularity uniformly on \([0,T]\) and the spatially semidiscrete solutions satisfy the temporal regularity assumptions above uniformly in \(\boldsymbol M\). 
Furthermore, the exact, semidiscrete, and discrete values remain in a common bounded pointwise range on which the reaction is uniformly Lipschitz. 
Because \(I_{\boldsymbol M}\Nn(v)=\Nn_M(I_{\boldsymbol M}v)\) holds exactly at the grid nodes, the spatial defect \(I_{\boldsymbol M}\Aa u-\Aa_M I_{\boldsymbol M}u\) acts exclusively as a forcing term. This defect is effectively controlled by the grid-uniform \(L^1\) bound of the spatially semidiscrete resolvent family \(P\) associated with \(\Aa_M\). 
Finally, comparing the exact and semidiscrete trajectories and applying the weighted temporal error estimate from the discrete analogue of Theorem~\ref{thm:temporal-consistency}, we conclude that for the initial condition \(U^0=I_{\boldsymbol M}u_0\),
\[
\|U^n-u(t_n)\|_{L^2}
\le
C\left(
h\,t_n^{\alpha_1-1}
+
M^{-(s-\beta)}
\right).
\]
\end{remark}

\section{Blockwise SOE with adaptive short--long memory splitting}
\label{sec:adaptive-block-soe}
Direct evaluation of the grid version of the history term in \eqref{eq:pde-cq-volterra} requires \(\mathcal O(N^2)\) vector operations and the storage of \(\mathcal O(N)\) vectors over \(N\) time steps. Fast CQ and SOE methods reduce this computational cost via blockwise integration or exponential recurrences \cite{SchadleLopezFernandezLubich2006,jiang2017fast,zhu2019fast}. We directly fit the CQ weights extracted from \eqref{eq:pde-fft-weights}. The most recent recovered weights are applied without any SOE approximation, whereas distant history terms are approximated using local SOE fits constrained by a prescribed history-defect tolerance. This compression strategy modifies only the history evaluation, leaving the recovered current-step coefficient \(\check\omega_0\) and the nonlinear solver unchanged.

\subsection{Short-memory splitting and blockwise compression}
Let \(U_{\rm e}^n\) denote the exact-weight, full-history Fourier--BE--CQ
solution and \(U_{\rm num}^n\) the computed solution with recovered weights,
SOE compression, and inexact Newton--GMRES solves, with
\(U_{\rm e}^0=U_{\rm num}^0=U^0\). Set
\[F_{\rm e}^j:=\Aa_M U_{\rm e}^j+\Nn_M(U_{\rm e}^j),
\qquad
F_{\rm num}^j:=\Aa_M U_{\rm num}^j+\Nn_M(U_{\rm num}^j),
\qquad 0\le j\le N.\]
For an additive source \(f(t)\), evaluated at the grid nodes, add \(f(t_j)\) to both
\(F_{\rm e}^j\) and \(F_{\rm num}^j\), and replace
\(\Aa_M V+\Nn_M(V)\) by \(\Aa_M V+\Nn_M(V)+f(t_n)\) in the
current-step equation and its residual. The same source values are used
in both schemes.

Choose \(0\le J_0<N\). Use the recovered weights \(\check\omega_\ell\) directly for \(\ell\le J_0\), and consider SOE compression only for \(\ell>J_0\).
Let \(\mathcal P=\mathcal P_{\rm cmp}\cup\mathcal P_{\rm ucmp}\)
be a disjoint partition of the integer lags \(\{J_0+1,\ldots,N\}\)
into compressed and uncompressed blocks. The computed history term is
\[
\widetilde H^n = U^0+\underbrace{\sum_{\ell=1}^{\min\{n,J_0\}}
\check\omega_\ell F_{\rm num}^{n-\ell}}_{H_{\rm short}^n}
+\sum_{I\in\mathcal P_{\rm cmp}}\widetilde H_I^n
+\sum_{I\in\mathcal P_{\rm ucmp}}H_I^n,
\qquad 1\le n\le N.
\]
where \(H_{\rm short}^n\) is the short-memory contribution.
The far-tail interval starts as one root block and is bisected adaptively.
The terms \(\widetilde H_I^n\) use SOE fits, whereas \(H_I^n\) use
\(\check\omega_\ell\) directly. Thus uncompressed terms have no SOE approximation error but still contain the coefficient-recovery error.

For a block $I=[I_a,I_b]$, define its width and recovered absolute CQ-weight mass by
\[
|I|=I_b-I_a+1,
\qquad
m_I=\sum_{\ell\in I}|\check\omega_\ell|.
\]
The initial block is assigned the history-compression tolerance
\(\tau_{\rm hist}>0\), and each block inherits a local tolerance \(\tau_I\).
When a block \(I\) with \(|I|>1\) is bisected at
\[
K=\left\lfloor\frac{I_a+I_b}{2}\right\rfloor,
\qquad
I_{\rm L}=[I_a,K],
\qquad
I_{\rm R}=[K+1,I_b],
\]
the tolerances of the two subblocks are assigned by
\begin{equation}
\label{eq:absoe-budget-split}
\tau_{I_{\rm L}}=\tau_I\frac{m_{I_{\rm L}}}{m_I},
\qquad
\tau_{I_{\rm R}}=\tau_I-\tau_{I_{\rm L}},
\qquad m_I>0.
\end{equation}
If \(m_I=0\), block widths are used in place of coefficient masses.
Each bisection preserves the sum of the local tolerances, and the tolerances assigned
to uncompressed terminal blocks are discarded. Consequently, $\sum_{I\in\mathcal P_{\rm cmp}}\tau_I\le\tau_{\rm hist}$.
Together with the local acceptance condition \eqref{eq:absoe-local-accept}, this gives
\[
F_{\rm bd}\sum_{I\in\mathcal P_{\rm cmp}}\rho_I
\le\sum_{I\in\mathcal P_{\rm cmp}}\tau_I
\le\tau_{\rm hist}.
\]

\subsection{Adaptive blockwise SOE approximation}
On \(I=[I_a,I_b]\), approximate the recovered CQ coefficients by
\[
\check\omega_\ell
 \approx
 \omega_\ell^{I,\mathrm{SOE}}
 :=
 \sum_{q=1}^{R_I}a_{I,q}z_{I,q}^{\,\ell},
 \qquad I_a\le\ell\le I_b,
\]
with prescribed positive decay parameters \(b_{I,q}\) and \(z_{I,q}=e^{-b_{I,q}}\in(0,1)\).
For a trial rank \(R_I\), $b_{I,q}\in\left[
 \frac{\chi_-}{I_b},
 \frac{\chi_+}{I_a}
 \right]$
with \(0<\chi_-<\chi_+\).
For \(R_I>1\), the nodes are logarithmically equispaced, including both endpoints:
\[
 b_{I,q}=\left(\frac{\chi_-}{I_b}\right)^{1-\frac{q-1}{R_I-1}}
          \left(\frac{\chi_+}{I_a}\right)^{\frac{q-1}{R_I-1}},
 \qquad q=1,\ldots,R_I.
\]
For \(R_I=1\), the geometric mean of the two
endpoints is used. This scales the exponential decay lengths to the lag interval of the block.

Let \(A^I_{\ell q}=e^{-\ell b_{I,q}}\) and
\(\check{\boldsymbol\omega}_I=(\check\omega_{I_a},\ldots,\check\omega_{I_b})^\top\).
For \(m_I>0\), define the fitting weights by
\[
 \vartheta_\ell^I
 =
 \frac{1}{
 \max\left\{
 |\check\omega_\ell|,
 10^{-10}\max_{i\in I}|\check\omega_i|
 \right\}},
 \qquad \ell\in I.
\]
For \(m_I=0\), set \(\vartheta_\ell^I=10^{10}\). The same regularized
solve gives \(\boldsymbol a_I=0\) and \(\rho_I=0\).
The coefficient vector is the
Tikhonov-regularized weighted least-squares solution
\[
 \boldsymbol a_I
 =
 \arg\min_{\boldsymbol a\in\mathbb R^{R_I}}
 \left\{
  \left\|
  (W_I)^{1/2}
 (A^I\boldsymbol a-\check{\boldsymbol\omega}_I)
 \right\|_2^2
 +
 \lambda_{\rm reg}\|\boldsymbol a\|_2^2
 \right\},\quad \lambda_{\rm reg}>0,
\]
where
\(
 W_I=\operatorname{diag}(\vartheta_{I_a}^I,\ldots,\vartheta_{I_b}^I)
\).
The hybrid weights balance the fit. Acceptance uses the unweighted
\(\ell^1\)-residual
\[
\rho_I
:=
\sum_{\ell=I_a}^{I_b}
\left|
\check\omega_\ell-\omega_\ell^{I,\mathrm{SOE}}
\right|.
\]
The prescribed bound \(F_{\rm bd}>0\) converts the coefficient
residual into a history-defect bound provided that
\begin{equation}
\label{eq:absoe-F-bound}
\max_{0\le j\le N}\|F_{\rm num}^j\|_{\boldsymbol M}
\le F_{\rm bd}.
\end{equation}
The periodic Fourier discretization satisfies
\(
\|\Aa_M V\|_{\boldsymbol M}
\le C\|V\|_{H_{\boldsymbol M}^{\beta}}.
\)
A uniform bound on this discrete Sobolev norm controls the diffusion term, while a uniform pointwise bound controls the polynomial reaction term.
Together with a bound on any additive source, these bounds yield a mesh-independent bound of the form \eqref{eq:absoe-F-bound}.
In the computation, \eqref{eq:absoe-F-bound} is checked directly on the computed trajectory.

A trial block is accepted when
\begin{equation}
\label{eq:absoe-local-accept}
F_{\rm bd}\rho_I\le\tau_I.
\end{equation}
For each block, the integer ranks \(R_{\min},R_{\min}+1,\ldots,R_{\max}\) are tested in increasing order, and the first rank satisfying
\eqref{eq:absoe-local-accept} is retained. If no tested rank passes this condition, the block is bisected when its width exceeds the prescribed
bisection cutoff, which is at least one lag. Otherwise it is retained without compression.

\paragraph{Compression error bound and right-hand-side check}
Let \(F_{\rm bd}^{(0)}>0\) be an initial estimate obtained without a full-history solution. For \(k\ge0\), construct the partition and local ranks
using \eqref{eq:absoe-budget-split} and \eqref{eq:absoe-local-accept}, with \(F_{\rm bd}=F_{\rm bd}^{(k)}\). The resulting defect bound is
\begin{equation}
\label{eq:fbd-design-certificate}
\mathcal C_{\rm des}^{(k)}
:=
F_{\rm bd}^{(k)}
\sum_{I\in\mathcal P_{\rm cmp}^{(k)}}\rho_I^{(k)}
\le \tau_{\rm hist}.
\end{equation}
After computing the resulting blockwise SOE trajectory, define its maximum right-hand-side norm and test
\begin{equation}
\label{eq:fbd-trajectory-certificate}
F_*^{{\rm num},(k)}
:=\max_{0\le j\le N}\|F_{\rm num}^{j,(k)}\|_{\boldsymbol M},
\qquad
F_*^{{\rm num},(k)}\le F_{\rm bd}^{(k)}.
\end{equation}
Together, \eqref{eq:fbd-design-certificate} and
\eqref{eq:fbd-trajectory-certificate} imply the sharper a posteriori bound
\[
\mathcal C_{\rm post}^{(k)}
:=
F_*^{{\rm num},(k)}
\sum_{I\in\mathcal P_{\rm cmp}^{(k)}}\rho_I^{(k)}
\le \tau_{\rm hist}.
\]
The computed history-compression defect in \eqref{eq:absoe-history-residual} is also checked against
\(\mathcal C_{\rm post}^{(k)}\) and \(\tau_{\rm hist}\).
If \(F_*^{{\rm num},(k)}>F_{\rm bd}^{(k)}\), the update
\begin{equation}
\label{eq:fbd-restart-update}
F_{\rm bd}^{(k+1)}
=2\max\{F_{\rm bd}^{(k)},F_*^{{\rm num},(k)}\}
\end{equation}
is followed by a new local block construction and time integration from
\(t=0\). 
A computed defect exceeding \(\mathcal C_{\rm post}^{(k)}\)
requires a separate check of the recurrence and residual evaluation.
The solution is accepted only after all three checks are satisfied.
If the restart limit is reached without satisfying the checks, the run is reported as unsuccessful. The prescribed restart limit and the accumulated fitting and time-stepping costs are reported in Section~\ref{sec:numerics}.
The successive-doubling weight-recovery check is handled separately, since changing \(F_{\rm bd}\) does not improve the recovered weights.

Set \(F_{\rm num}^j=0\) for \(j<0\).  For
\(I=[I_a,I_b]\in\mathcal P_{\rm cmp}\) and \(1\le q\le R_I\), define
\[
 Y_{I,q}^{n}
 =
 \sum_{\ell\in I\cap\{1,\ldots,n\}}
 z_{I,q}^{\,\ell}F_{\rm num}^{n-\ell},
 \qquad Y_{I,q}^0=0,
\]
for $1\le n\le N$. Direct re-indexing gives the recurrence
\begin{equation}
\label{eq:absoe-block-recurrence}
Y_{I,q}^{n}
 =
 z_{I,q}Y_{I,q}^{n-1}
 +z_{I,q}^{{I_a}}F_{\rm num}^{n-{I_a}}
 -z_{I,q}^{{I_b}+1}F_{\rm num}^{n-1-{I_b}}.
\end{equation}
Hence the compressed contribution of block \(I\) is $ \widetilde H_I^n = \sum_{q=1}^{R_I}a_{I,q}Y_{I,q}^{n}.$
For an uncompressed block \(I\in\mathcal P_{\rm ucmp}\),
\[
H_I^n = \sum_{\ell\in I\cap\{1,\ldots,n\}} \check\omega_\ell F_{\rm num}^{n-\ell}.
\]
At time \(t_n\), the history uses only \(F_{\rm num}^j\) with
\(0\le j<n\). The current-step equation is
\begin{equation}
\label{eq:absoe-current-equation}
V
=
\widetilde H^n
+
\check\omega_0\bigl(\Aa_M V+\Nn_M(V)\bigr),
\qquad 1\le n\le N.
\end{equation}
The computed value \(U_{\rm num}^n\) is obtained by solving
\eqref{eq:absoe-current-equation} inexactly.

\subsection{Stability transfer and error estimate}

For the final computed scheme, the blockwise SOE compression residual is
\begin{equation}
\label{eq:absoe-history-residual}
 r_{\rm hist}^{n}
 :=
 \sum_{I\in\mathcal P_{\rm cmp}}
  \sum_{\ell\in I\cap\{1,\ldots,n\}}
  \left(
 \omega_\ell^{I,\mathrm{SOE}}-\check\omega_\ell
  \right)
 F_{\rm num}^{n-\ell},
 \qquad 1\le n\le N,
\end{equation}
with \(r_{\rm hist}^0=0\). Suppress the accepted-attempt index and set
\(\widehat\varepsilon_{\rm hist}:=
\max_{0\le n\le N}\|r_{\rm hist}^n\|_{\boldsymbol M}\).
The triangle inequality gives
\[
\widehat\varepsilon_{\rm hist}
\le F_*^{\rm num}\sum_{I\in\mathcal P_{\rm cmp}}\rho_I
\le F_{\rm bd}\sum_{I\in\mathcal P_{\rm cmp}}\rho_I
\le\sum_{I\in\mathcal P_{\rm cmp}}\tau_I
\le\tau_{\rm hist}.
\]
This bound uses the right-hand-side bound, the local acceptance condition
\eqref{eq:absoe-local-accept}, and preservation of the sum of the local tolerances in \eqref{eq:absoe-budget-split}. Uncompressed terms contribute no SOE defect.
The ratio \(\frac{F_*^{\rm num}}{F_{\rm bd}}\le1\) gives \(\widehat\varepsilon_{\rm hist}\le \frac{F_*^{\rm num}}{F_{\rm bd}}\tau_{\rm hist}\). Measuring
the computed right-hand-side norm sharpens the prescribed bound without replacing either acceptance condition.

For the computed sequence, the Cauchy--FFT coefficient-recovery residual sequence is
$r^n_{\rm weight} = \sum_{j=0}^{n}
(\check\omega_{n-j}-\omega_{n-j})F_{\rm num}^j$ for $1\le n\le N$,
with $r^0_{\rm weight}=0$. By taking maximum over the time interval, we define the weight-recovery defect by
\[
\varepsilon_{\rm weight}
:= \max_{0\le n\le N}
\left\| r^n_{\rm weight}\right\|_{\boldsymbol M}.
\]
The exact coefficients \(\omega_\ell\) are unavailable.
For the last two recovery levels \(L\) and \(2L\), compute
\[
 \widehat\varepsilon_{\rm weight}
 :=
 F_*^{\rm num}
 \sum_{\ell=0}^{N}
 \left|
 \check\omega_\ell^{(2L)}
 -\check\omega_\ell^{(L)}
 \right|.
\]
This diagnostic compares two recovered coefficient sets and does not by itself bound \(\varepsilon_{\rm weight}\). 
The coefficient-recovery stopping test remains the one specified earlier. 
At the current step, set
\[
\widetilde G_n(V)
 :=
 V-\widetilde H^n
 -\check\omega_0\bigl(\Aa_M V+\Nn_M(V)\bigr),
 \qquad 1\le n\le N,
\]
and define \(r^n_{\rm alg}=\widetilde G_n(U_{\rm num}^{n})\)
for \(1\le n\le N\), with \(r^0_{\rm alg}=0\). The algebraic defect is
\[ \varepsilon_{\rm alg}
:=
\max_{0\le n\le N}
\|r^n_{\rm alg}\|_{\boldsymbol M}.\]

\begin{theorem}
\label{thm:absoe-compression-transfer}
Assume that the hypotheses of the grid version of Theorem~\ref{thm:fully-discrete-stability}
hold for \(U_{\rm e}\) and \(U_{\rm num}\), uniformly with respect to the spatial grid. 
Suppose that \eqref{eq:absoe-F-bound} holds, and that every compressed block satisfies \eqref{eq:absoe-local-accept}.
Then
\[
\max_{0\le n\le N}
\|U_{\rm num}^n-U_{\rm e}^n\|_{\boldsymbol M}
\le
C_T(1+\log(N+1))
\left(
\tau_{\rm hist}
+\varepsilon_{\rm weight}
+\varepsilon_{\rm alg}
\right),
\]
where \(C_T\) is independent of \(h\) and the spatial grid.
\end{theorem}

\begin{proof}
Relative to the exact-weight BE--CQ recursion, the computed solution satisfies
\[
U_{\rm num}^n
=
U^0+
\sum_{j=0}^n
\omega_{n-j}
F_{\rm num}^j+r_{\rm ext}^n,
\qquad 1\le n\le N,
\]
where
\(r_{\rm ext}^n=r_{\rm hist}^n+r_{\rm weight}^n+r_{\rm alg}^n\)
and $r_{\rm ext}^0=0.$
The history-defect bound above and the residual definitions give
\[
\max_{0\le n\le N}\|r_{\rm ext}^n\|_{\boldsymbol M}
\le
\tau_{\rm hist}
+\varepsilon_{\rm weight}
+\varepsilon_{\rm alg}.
\]
Applying the grid version of Theorem~\ref{thm:fully-discrete-stability} to this residual proves the result.
\end{proof}

Combining Theorem~\ref{thm:absoe-compression-transfer} with the grid version of Theorem~\ref{thm:temporal-consistency} shows that positive-time first-order accuracy is preserved if
\[
\tau_{\rm hist}+\varepsilon_{\rm weight}+\varepsilon_{\rm alg}
=\mathcal O\!\left(\frac{h}{1+\log(N+1)}\right).
\]
In convergence tests, \(\tau_{\rm hist}\) is scaled accordingly.

\subsection{Nonlinear solution, complexity, and storage}
Apply the damped inexact Newton--GMRES method of Section~\ref{subsec:fourier_newton_solver} to \(\widetilde G_n\), with \(H^n,\omega_0\) replaced by \(\widetilde H^n,\check\omega_0\).
The corresponding current-step conditions are evaluated with
\(\check\omega_0\).
The outer iteration is terminated when
\[
\|\widetilde G_n(V_k)\|_{\boldsymbol M}
\le
{\rm atol}
+
{\rm rtol}\max\{1,\|\widetilde H^n\|_{\boldsymbol M}\}.
\]
The resulting maximum nonlinear residual is the algebraic defect $\varepsilon_{\rm alg}$ entering the stability-transfer estimate.

\paragraph{Computational complexity and storage}

Define the effective history size by
\begin{equation}
\label{eq:absoe-effective-history-size}
K_{\rm work}
=J_0+\sum_{I\in\mathcal P_{\rm ucmp}}|I|
+\sum_{I\in\mathcal P_{\rm cmp}}R_I.
\end{equation}
The recurrence-based history evaluation for one trajectory requires $\mathcal O(NK_{\rm work})$ vector operations. It reduces the \(\mathcal O(N^2)\) direct-history
work when \(K_{\rm work}\ll N\).
This count excludes coefficient recovery, fitting, local block construction,
restarts, and any direct evaluation of history defects for diagnostics.
The costs of all attempts are accounted for separately.

The exponential sums require
\(\mathcal O(\sum_{I\in\mathcal P_{\rm cmp}}R_I)\) grid vectors.
The fixed-lag recurrence \eqref{eq:absoe-block-recurrence} also uses
\(F_{\rm num}^{n-1-I_b}\), so the current implementation retains the
past right-hand-side values and still requires \(\mathcal O(N)\) stored
vectors. The full \(U_{\rm num}\)-trajectory is stored for diagnostics
and output, not for the recurrence. Thus
\(K_{\rm work}\) measures the number of terms used in history evaluation and the associated arithmetic cost,
rather than the asymptotic total storage.

\section{Numerical results and simulations}
\label{sec:numerics}
The experiments assess temporal accuracy, the perturbation introduced by history compression, and computational cost.
They also examine the history-defect bounds and verify the compression and right-hand-side acceptance conditions.

We compare full-history BE--CQ (BE--FH) with the adaptive blockwise SOE method (Block--BE--SOE). Both runs use the same recovered CQ weights, time steps, and nonlinear solver tolerances. 
The maximum trajectory discrepancy is
\[
D_\infty:=
\max_{0\le n\le N}
\|\boldsymbol U_{\rm SOE}^n-\boldsymbol U_{\rm FH}^n\|,
\]
where \(\|\cdot\|\) denotes the Euclidean norm for the
stiff reaction system and the discrete \(L^2\) norm for the PDEs.

All experiments use \(({\rm rtol},{\rm atol})=(10^{-10},10^{-12})\) and at most 30 Newton iterations per time step.
The forcing parameters are selected adaptively in \([10^{-6},0.1]\), with \(\eta_0=0.1\) at each step.
For the PDEs, GMRES uses restart length 30 and at most 300 Krylov iterations per linear solve.
For the stiff reaction system, it is unrestarted with at most 20 iterations.
Armijo backtracking uses parameter \(10^{-4}\) and at most 40 halvings of a unit step.

Following Section~\ref{subsec:cauchy-fft}, the Cauchy–FFT recovery begins with $\rho_r=0.9995$, which is increased to $(1+\rho)/2$ if necessary to maintain $\rho<\rho_r<1$. 
The recovery parameters include $F_s=10^{-2}$, a target accuracy of $10^{-12}$, $4096$ circle samples, and an initial FFT length of at least $16(N+1)$. 
The successive-doubling tolerance is set to $10^{-11}$ for at most four doublings, and a failed doubling check stops the computation.
The local SOE fits use hybrid weighted least squares with
\(\lambda_{\rm reg}=10^{-14}\) and candidate nodes in
\([10^{-1}/I_b,1/I_a]\). For each PDE system, the independent initial predictor is
\[
F_{\rm bd}^{(0)}
=
\max\left\{F_0,\|\Aa_M U^0+\Nn_M(U^0)\|_{\boldsymbol M}\right\}.
\]

\subsection{Stiff reaction system}

We first consider the 20-component system
\[
{}^{S}D_0^{\alpha(t)}\boldsymbol u(t)
=
-\Lambda\boldsymbol u(t)-\boldsymbol u(t)^3
+\boldsymbol f(t),
\qquad 0<t\le200,
\]
where
\(\Lambda=\operatorname{diag}(10^{2(j-1)/19})_{j=1}^{20}\),
\((\alpha_1,\alpha_2,c)=(0.4,0.9,10^{-2})\),
and the powers are componentwise.
The manufactured exact solution is
\(\boldsymbol u_*(t)=\boldsymbol A e^{-0.05t}\),
with \(A_j=0.4+0.20(j-1)\) and
\(\boldsymbol u(0)=\boldsymbol A\).
The corresponding forcing is
\[
 \boldsymbol f(t)=-0.05\boldsymbol A q_{0.05}(t)
 +\Lambda\boldsymbol u_*(t)+\boldsymbol u_*(t)^3,
 \qquad
 q_\nu=\LT^{-1}\!\left\{\frac{z^{B(z)-1}}{z+\nu}\right\}.
\]

A 64-node midpoint rule on the modified Talbot contour \cite{DingfelderWeideman2015} is used for \(q_{0.05}\), with its value at \(t=0\) taken as the right limit \(q_{0.05}(0)=0\). 
We take \(h=2^{-3},\ldots,2^{-8}\), so that \(N=1600,\ldots,51200\), and set
\[
 \begin{gathered}
 J_0=32,\qquad
 \tau_{\rm hist}=\frac{0.05h}{1+\log(N+1)},\qquad F_0 = 1.
 \end{gathered}
\]
By this choice of \(\tau_{\rm hist}\), the history-compression term has the first-order scale required by Theorem~\ref{thm:absoe-compression-transfer}. 
Block--BE--SOE uses \(1\le R_I\le30\), a three-lag bisection cutoff, and at most three restarts.
All terminal blocks are compressed, so \(K_{\rm work}=32+\sum_I R_I\).

We measure the Block--BE--SOE error by
\[
 E_2^{\rm B}
 =\left(h\sum_{n=1}^{N}
 \|\boldsymbol U_{\rm B}^n-\boldsymbol u_*(t_n)\|_2^2\right)^{1/2}.
\]

The reported computational times \(\mathcal T_{\rm FH}\) and
\(\mathcal T_{\rm B}\) include weight recovery and implicit time stepping, with SOE fitting additionally included for Block--BE--SOE, each measured over three runs.
Repeated fitting and time stepping are included in their respective phase totals. 
The diagonal reaction Jacobian is preconditioned exactly,
so the preconditioned linear operator is the identity.

\begin{table}[!htbp]
\centering
\scriptsize
\setlength{\tabcolsep}{2.5pt}
\caption{Block--BE--SOE and BE--FH  for the stiff reaction system: temporal accuracy, trajectory
discrepancies, normalized history-error bounds and measured defects, and CPU time (seconds).}
\label{tab:num-convergence}
\resizebox{\linewidth}{!}{%
\begin{tabular}{@{}ccccccc@{}}
\toprule
\(N\)
& \(E_2^{\rm B}\) (rate)
& \(D_\infty\)
& \(\mathcal C_{\rm des}/\tau_{\rm hist}\)
& \(\mathcal C_{\rm post}/\tau_{\rm hist}\)
& \(\widehat\varepsilon_{\rm hist}/\tau_{\rm hist}\)
& \(\mathcal T_{\rm FH}/\mathcal T_{\rm B}\) (speedup)\\
\midrule
1600 & \(1.739{\times}10^{-3}\) (--) & \(1.06{\times}10^{-6}\) & 0.6431 & 0.3212 & 0.0347 & 0.107/0.151 (0.707)\\
3200 & \(8.791{\times}10^{-4}\) (0.984) & \(2.68{\times}10^{-7}\) & 0.4459 & 0.2229 & 0.0175 & 0.348/0.208 (1.676)\\
6400 & \(4.434{\times}10^{-4}\) (0.987) & \(4.79{\times}10^{-8}\) & 0.4234 & 0.2116 & 0.0064 & 0.887/0.371 (2.388)\\
12800 & \(2.232{\times}10^{-4}\) (0.990) & \(4.96{\times}10^{-8}\) & 0.6366 & 0.3183 & 0.0143 & 4.451/0.746 (5.970)\\
25600 & \(1.122{\times}10^{-4}\) (0.993) & \(4.63{\times}10^{-8}\) & 0.8966 & 0.4483 & 0.0285 & 17.224/1.408 (12.230)\\
51200 & \(5.630{\times}10^{-5}\) (0.994) & \(2.05{\times}10^{-8}\) & 0.8863 & 0.4431 & 0.0258 & 63.197/2.925 (21.607)\\
\bottomrule
\end{tabular}}
\end{table}

\begin{figure}[!htbp]
\centering
\includegraphics[width=0.37\linewidth]{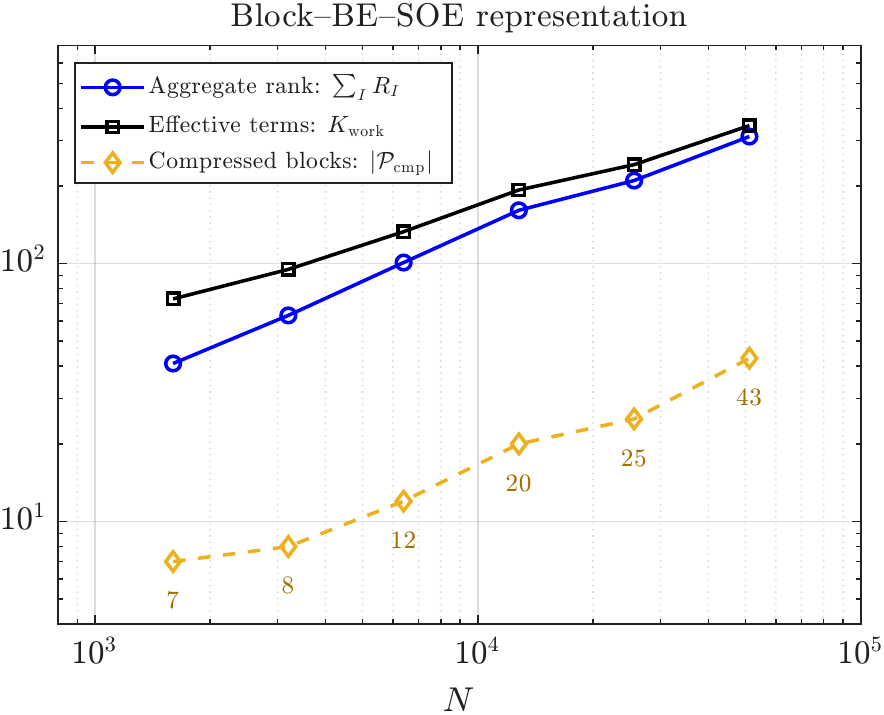}
\hspace{0.01\linewidth}
\includegraphics[width=0.37\linewidth]{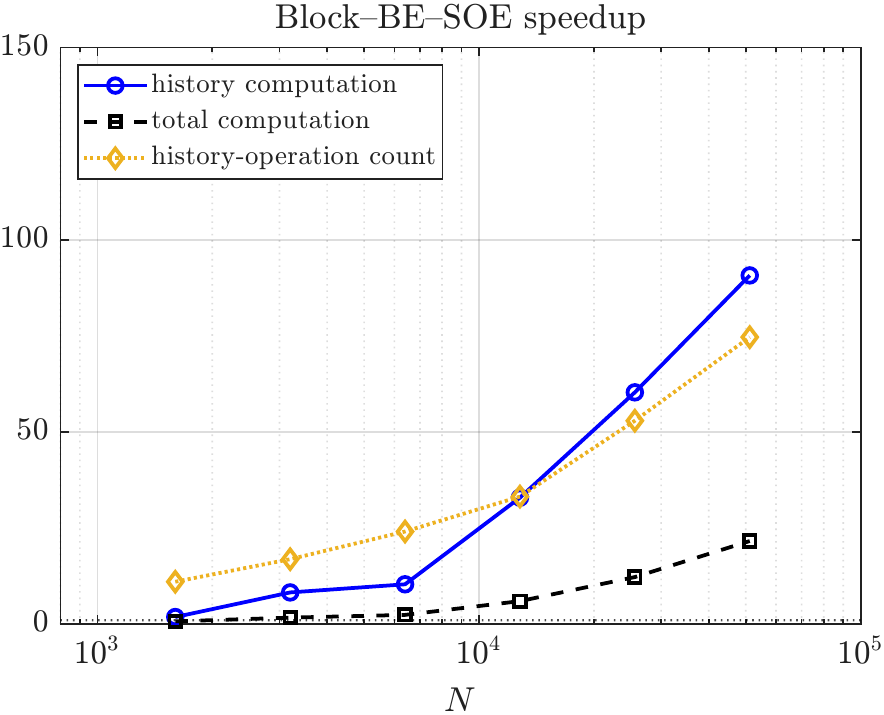}
\caption{Stiff reaction system at $T=200$.
Left: effective history size $K_{\rm work}$, total SOE rank,
and compressed-block count.
Right: BE--FH/Block--BE--SOE ratios for history time,
computational time, and history-operation count.}
\label{fig:num-block-eh}
\end{figure}

Table~\ref{tab:num-convergence} shows first-order convergence
of Block--BE--SOE, with
\(\widehat\varepsilon_{\rm hist}\le\mathcal C_{\rm post}
\le\mathcal C_{\rm des}\le\tau_{\rm hist}\) on all six grids.
The compression perturbation in the temporal \(L^2\) norm
remains below \(0.9\%\) of the BE--FH error.
As \(N\) increases 32-fold, \(K_{\rm work}\) grows from \(73\)
to \(343\).
Block--BE--SOE is faster from \(N=3200\), reaching a total
speedup of \(21.61\) at \(N=51200\).
The reduced effective history size lowers arithmetic work,
while the fixed-lag recurrences still require
\(\mathcal O(N)\) stored right-hand-side vectors.

The evaluation comparison in Figure~\ref{fig:num-block-eh} uses the same right-hand-side sequences for both methods to measure the history computation alone. The history speedup reaches a factor of roughly \(90\) on the finest mesh. Because the effective history size grows much more slowly than the total step number \(N\), the exponential recurrences yield increasing computational savings over direct summation. These specific savings translate directly into a significant reduction in the total computational time, including weight recovery, exponential fitting, and the implicit time stepping.

\paragraph{Comparison with single-tail SOE}
The unpartitioned approach, denoted as Single--BE--SOE, fits the entire tail $[J_0+1,N]$ using candidate decay parameters in $[10^{-4},2]$, tested ranks $1\le R\le300$, and an effective history size of $J_0+R$.
Applying the same fitting and acceptance procedures, we compare its performance against Block–BE–SOE. To systematically evaluate the tradeoff between direct summation and tail compression, we vary $J_0\in\{32,64,128\}$ at $N=12800$ and $25600$, keeping $\tau_{\rm hist}$ fixed for each $N$.

\begin{figure}[!htbp]
\centering
\makebox[\linewidth][c]{%
\includegraphics[width=0.33\linewidth]{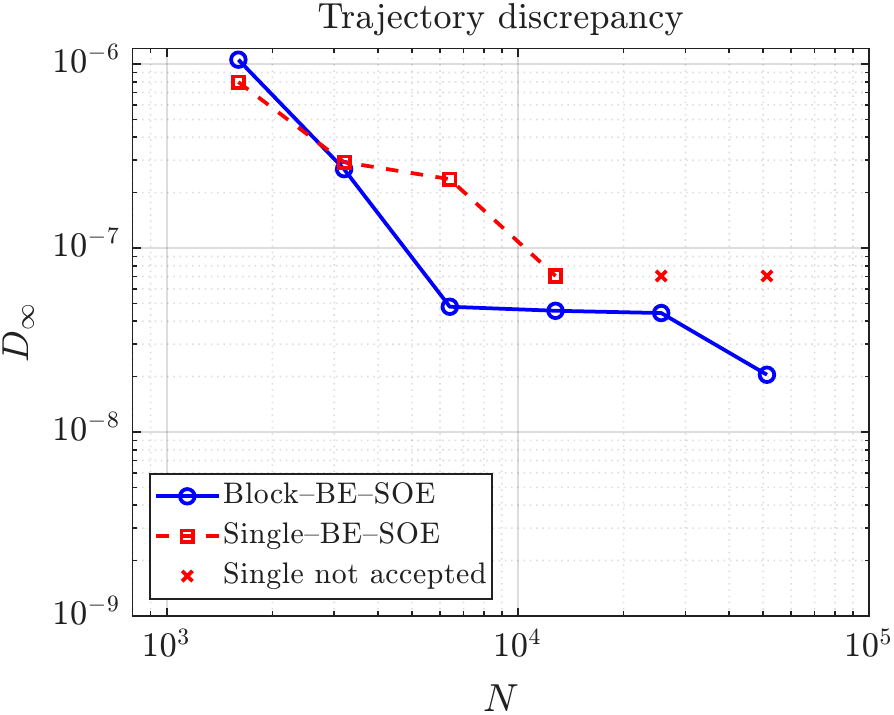}%
\hspace{0.02\linewidth}%
\includegraphics[width=0.33\linewidth]{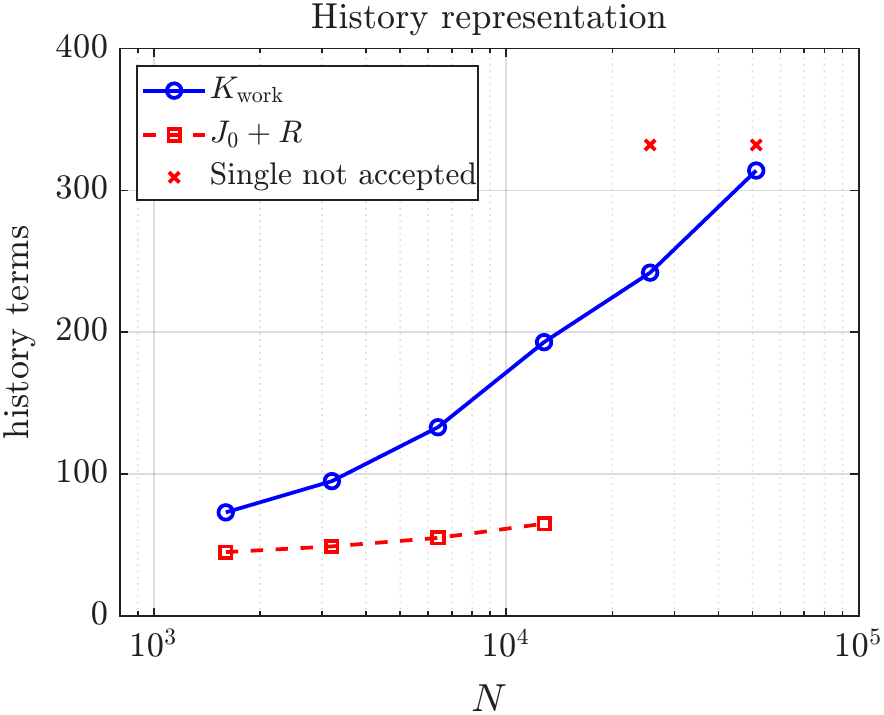}%
\hspace{0.02\linewidth}%
\includegraphics[width=0.325\linewidth]{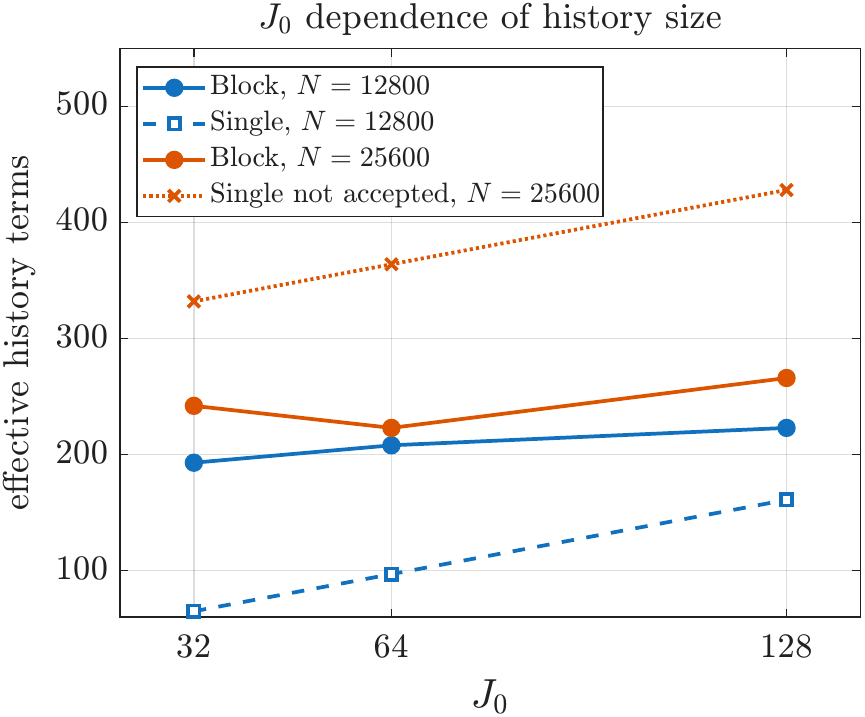}%
}
\caption{Blockwise and single-tail SOE compression for the stiff reaction system. Left: maximum trajectory discrepancy \(D_\infty\)
relative to BE--FH. Middle: effective history sizes \(K_{\rm work}\)
and \(J_0+R\). Both panels use \(J_0=32\).
Right: effective history size versus \(J_0\) at
\(N=12800\) and \(25600\).}
\label{fig:num-block-single}
\end{figure}

As shown in Figure~\ref{fig:num-block-single}, the single-tail SOE requires fewer history terms on the four coarser grids, whereas the blockwise compression successfully satisfies the prescribed tolerance across all six grids. 
At \(N=25600\) and \(51200\), none of the tested single-tail fits meets the coefficient-residual condition.
In the right panel, both history sizes increase with \(J_0\) at \(N=12800\).
At \(N=25600\), the blockwise \(K_{\rm work}\) is smallest at
\(J_0=64\) among the tested cutoffs, whereas no tested single-tail fit is accepted at any of the three cutoffs.
A larger \(J_0\) reduces per-step history work only when the rank saving exceeds the number of additional directly summed lags.

\subsection{Two-dimensional Nagumo equation}

The Nagumo equation is a canonical bistable reaction--diffusion model with
stable states \(u=0\) and \(u=1\), separated by the unstable threshold
\(u=a\) \cite{Nagumo1962}. Its fractional generalization is 
\[
 {}^{S}D_0^{\alpha(t)}u
 =-d_u(-\Delta)^{\beta/2}u
 +r_u u(1-u)(u-a),\qquad (x,y,t)\in(0,1)\times(0,1)\times(0,T],
\]
with
\(
 u_0(x,y)=\tfrac12
 \exp\!\left[-100\bigl((x-\tfrac12)^2+(y-\tfrac12)^2\bigr)\right].
\)
The parameters are
\((\alpha_1,\alpha_2,c)=(0.6,0.9,0.1)\),
\(d_u=2.5\times10^{-4}\), \(r_u=15\), \(a=0.35\), \(T=30\), and
\(h=2^{-8}\).  A \(256\times 256\) Fourier grid is used for
\(\beta=1.6,1.8,2\).
The SOE parameters are \(J_0=32\),
\(4\le R_I\le30\), a block-bisection cutoff of \(2\) lags, and
\(\tau_{\rm hist}=10^{-3}h\) with \(F_0 = 0.5\).

Figure~\ref{fig:num-ac-performance} tests compression accuracy
and cost.  At \(\beta=2\), Block--BE--SOE closely reproduces the BE--FH trajectory over \(0\le t\le30\), with
\(D_\infty=4.694\times10^{-9}\), while reducing the
reported computational time by a factor of \(6.68\).
The 15 compressed blocks give \(K_{\rm work}=155\) for \(N=7680\).
The computational time decreases from \(3489.407\) s to \(522.471\) s, including \(0.353\) s of precomputation for the block method.
Thus the history savings remain substantial after precomputation is included.
Both methods averaged approximately \(3\)
Newton and \(8\) total GMRES iterations per time step,
with maxima of \(3\) and \(9\), respectively.
\begin{figure}[htbp]
\centering
\includegraphics[width=0.39\linewidth]{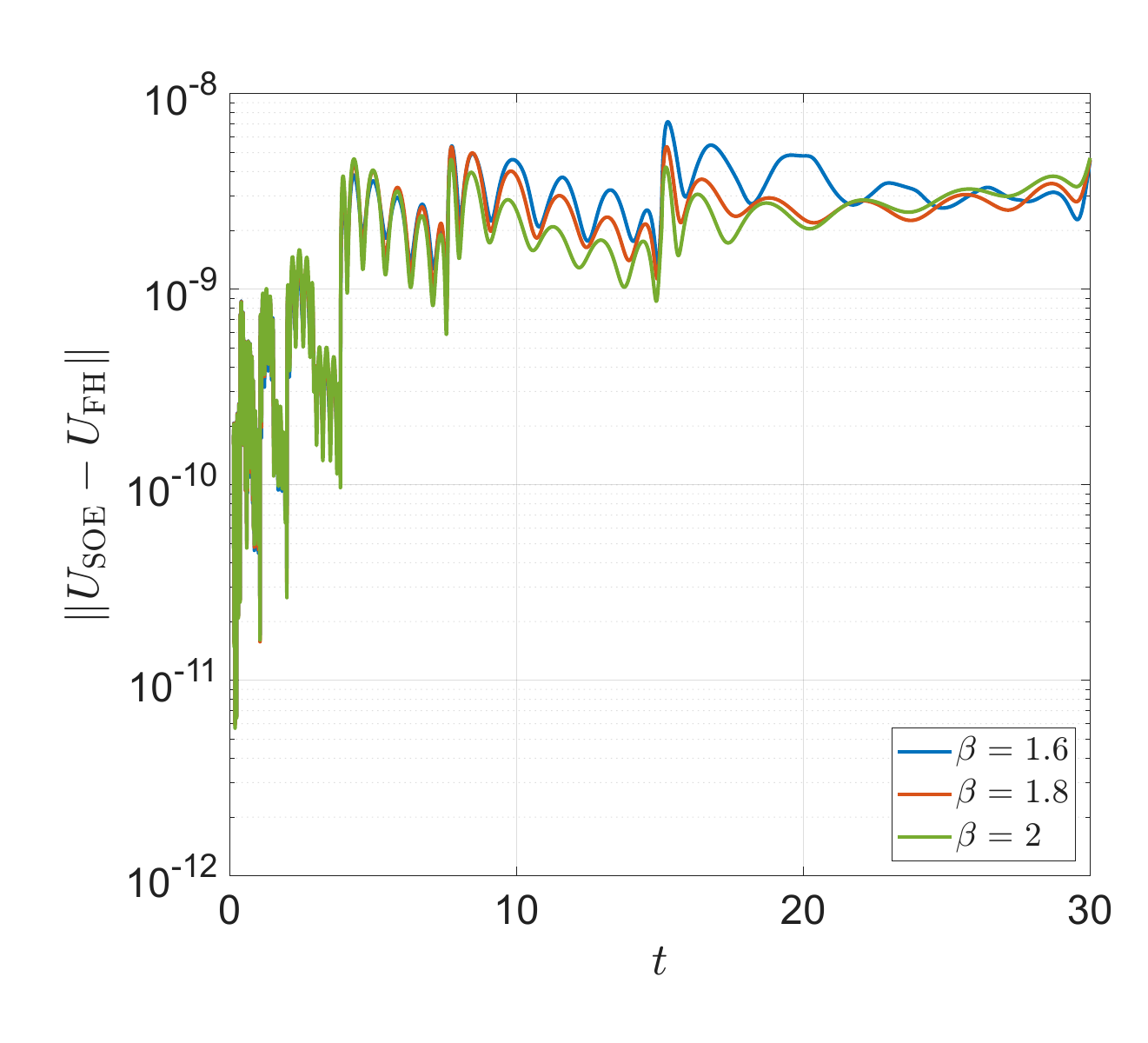}%
\hspace{0.03\linewidth}%
\includegraphics[width=0.39\linewidth]{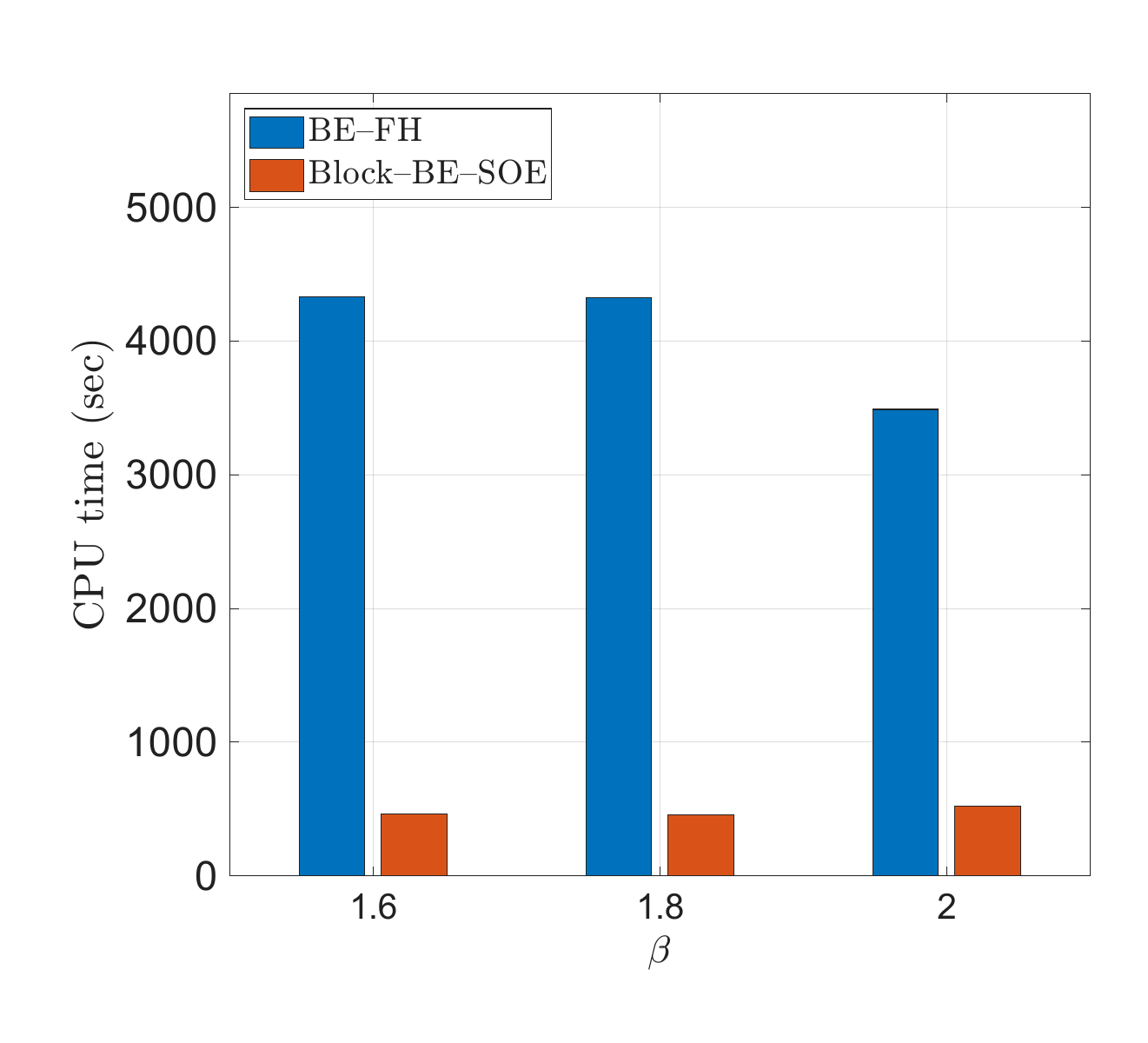}
\caption{Nagumo equation at \(\beta=1.6,1.8,2\) and \(T=30\).
Left: discrete \(L^2\) discrepancy between Block--BE--SOE and BE--FH.
Right: CPU time (s) including Block--BE--SOE precomputation.}
\label{fig:num-ac-performance}
\end{figure}

Figure~\ref{fig:num-ac-patterns} isolates the effect of the spatial order $\beta$. As $\beta$ decreases, the advancing plateau becomes sharper and more localized. This occurs because the fractional order modifies the Fourier diffusion multiplier $d_u\vert{}\kappa\vert{}^\beta$ nonuniformly across wavenumbers.
\begin{figure}[htbp]
\centering
\includegraphics[width=0.30\linewidth]{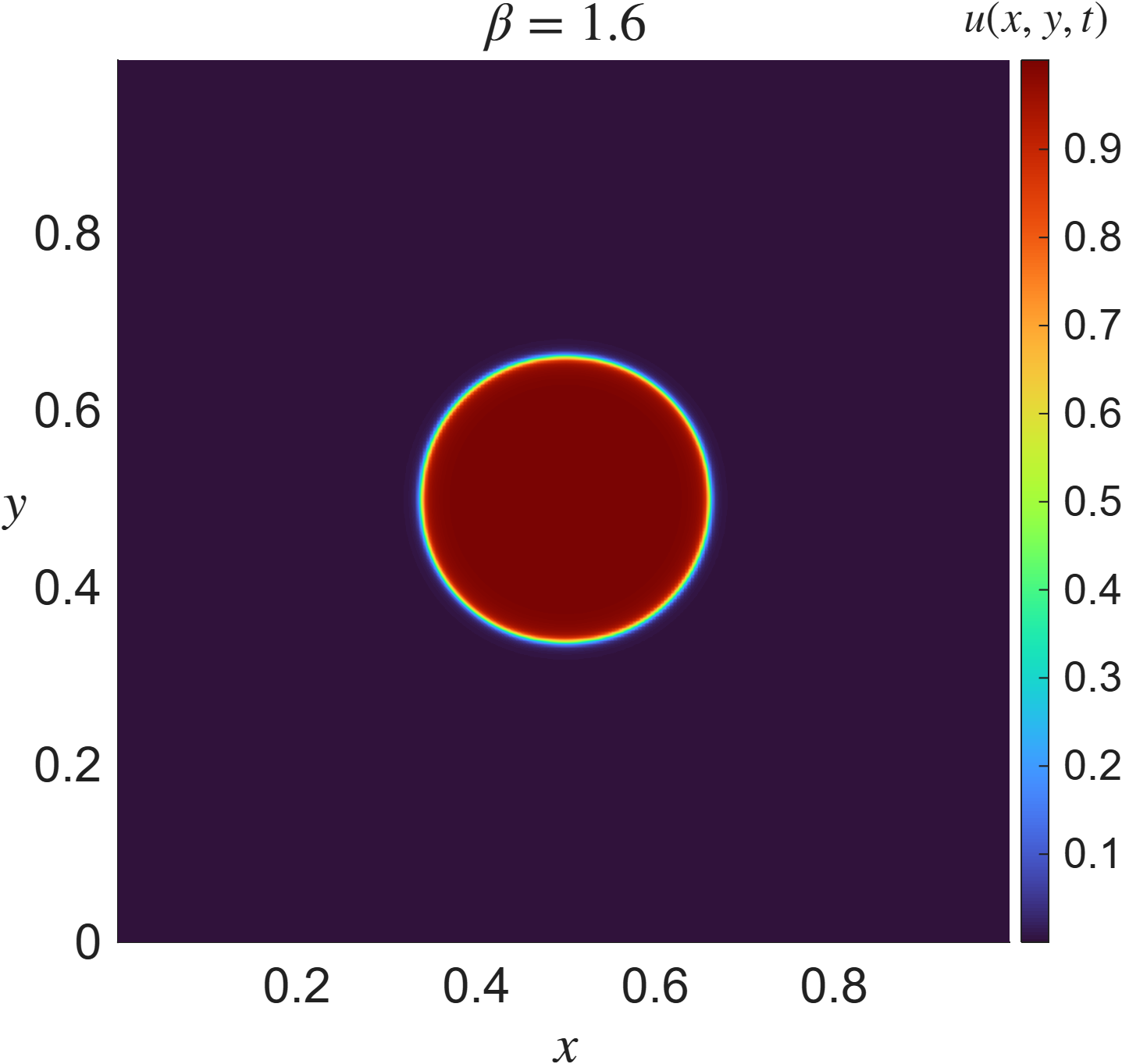}
\hfill
\includegraphics[width=0.30\linewidth]{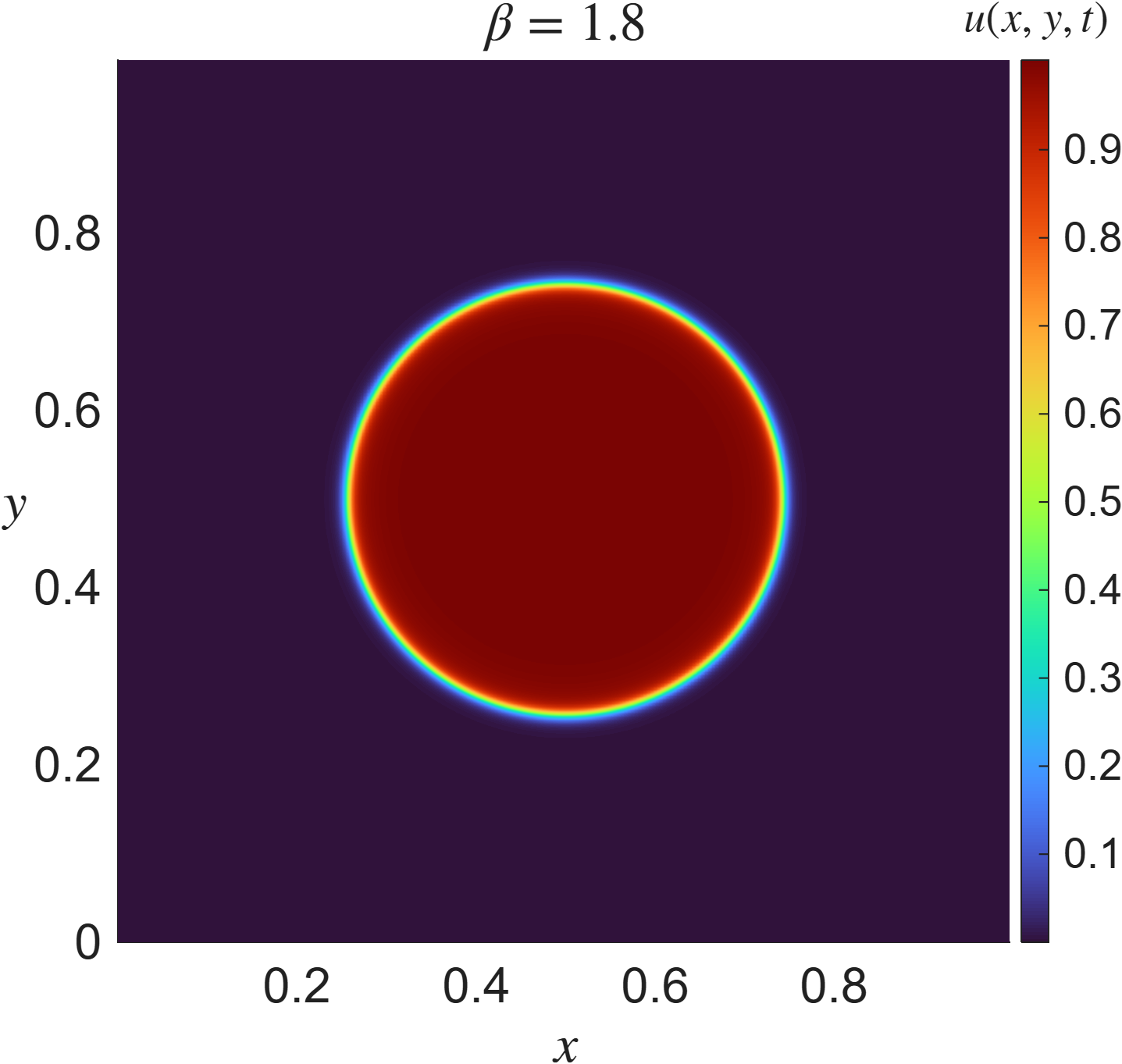}
\hfill
\includegraphics[width=0.30\linewidth]{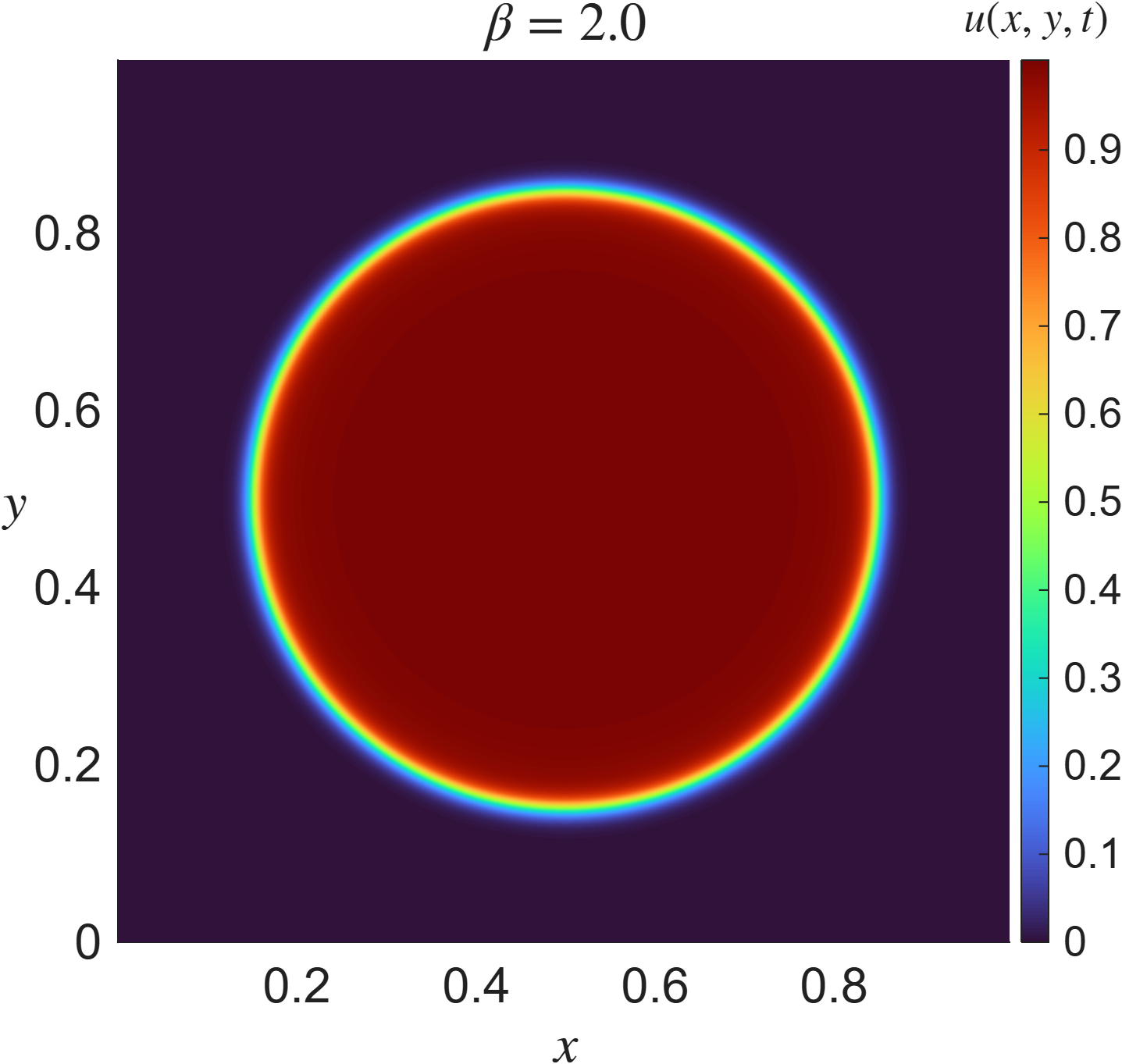}
\caption{Nagumo solutions \(u\) at \(T=30\) computed by Block--BE--SOE for \(\beta=1.6,1.8,2\).}
\label{fig:num-ac-patterns}
\end{figure}

\subsection{Two-dimensional FitzHugh--Nagumo system}

The FitzHugh--Nagumo model is a two-variable reduction of excitable membrane
dynamics, coupling a fast activator to a slower recovery variable
\cite{FitzHugh1961,Nagumo1962}.  The system is 
\[
\left\{\begin{aligned}
 {}^{S}D_0^{\alpha(t)}u
 &=-d_u(-\Delta)^{\beta/2}u+u(u-a)(1-u)-v\\
 {}^{S}D_0^{\alpha(t)}v
 &=-d_v(-\Delta)^{\beta/2}v
 +r_v(bu-\gamma v)
\end{aligned}\right.
,\quad(x,y,t)\in(0,1)\times(0,1)\times(0,T],
\]
where
\((\alpha_1,\alpha_2,c)=(0.6,0.9,0.1)\),
\((d_u,d_v,r_v)=(10^{-5}, 10^{-8}, 10^{-2})\), and
\((b,\gamma,a)=(0.5,1,0.1)\).  
On a $128\times128$ spatial grid, we set the time-step to $h=2^{-3}$ and integrate the system up to $T=1000$, which corresponds to $N=8000$ time steps.
We initialize $u=1$ in the lower-left quarter and $v=0.1$ in the upper half, while both are zero elsewhere. We then apply a single smoothing filter with a normalized periodic $3\times3$ Gaussian kernel of standard deviation $0.5$ grid spacings.
This standard perturbation reliably generates curved excitation fronts and spiral patterns in spatially fractional FitzHugh–Nagumo systems, see \cite{ZhangLiChenZhang2021}. Finally, the SOE parameters are chosen as $J_0=64$, $5\le R_I\le20$, a block-bisection cutoff of $2$ lags, and $\tau_{\rm hist}=10^{-3}h$ with $F_0 = 0.1$.

Figure~\ref{fig:num-fhn-performance} compares BE--FH and Block--BE--SOE over \(8000\) time steps to evaluate the accumulated effect of history compression. For the fractional orders \(\beta=1.6, 1.8\), 
and \(2\), the joint discrete spatial \(L^2\) discrepancies between the two trajectories are \(1.156\times10^{-5}\,,\, 4.270\times10^{-5}\), and \(2.174\times10^{-5}\), respectively. 
At \(\beta=2\), the block method requires only \(110\) effective history terms and reduces the total computational time from \(2169.598\) to \(215.703\) seconds, which achieves a \(10.06\)-fold speedup that fully accounts for \(0.053\) seconds of precomputation. 
The accepted local fits and verified right-hand-side bounds guarantee \(\widehat\varepsilon_{\rm hist}\le\tau_{\rm hist}\) across all three runs (Table~\ref{tab:posteriori defect}), providing the exact history-defect control required by Theorem~\ref{thm:absoe-compression-transfer}.
Across the three values of \(\beta\), both methods averaged approximately \(3\) Newton and \(7\)--\(8\) total GMRES iterations per time step, with maxima of \(3\) and \(9\), respectively.

\begin{figure}[htbp]
\centering
\includegraphics[width=0.39\linewidth]{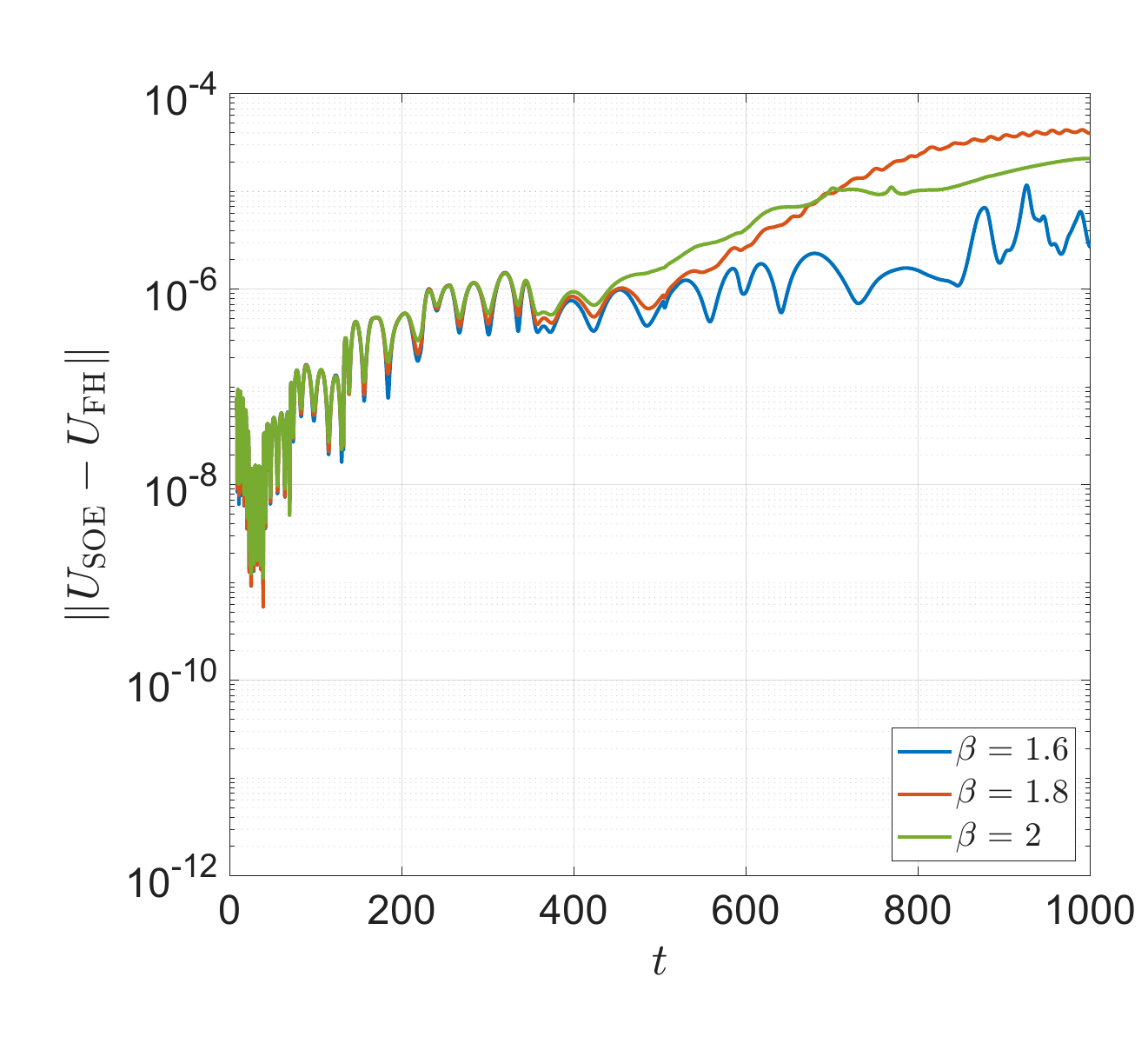}%
\hspace{0.03\linewidth}%
\includegraphics[width=0.39\linewidth]{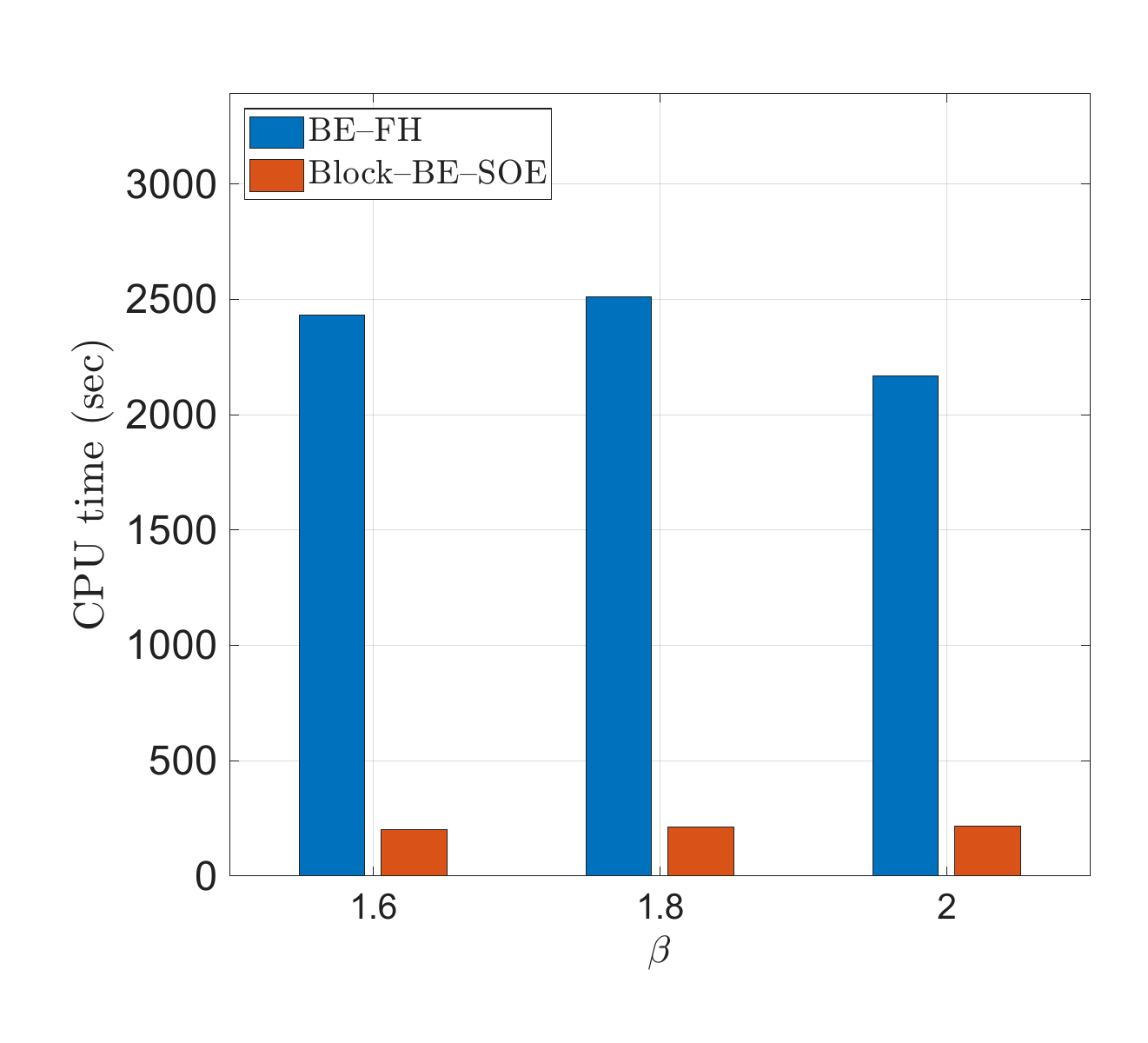}
\caption{FitzHugh--Nagumo system at \(\beta=1.6,1.8,2\) and
\(T=1000\).
Left: joint discrete \(L^2\) discrepancy for both components.
Right: CPU time (s) including Block--BE--SOE precomputation.}
\label{fig:num-fhn-performance}
\end{figure}

At \(T=1000\), the recovery field is concentrated in narrow bands for \(\beta=1.6\), whereas curved fronts and spiral segments occupy more of the periodic cell for \(\beta=1.8\) and \(2\), see Figure~\ref{fig:num-fhn-patterns}.
The fronts are broader at \(\beta=2\) than at \(\beta=1.8\). 
Such spiral structures arise through the interaction of excitation and slow recovery, with the free ends of excitation fronts curling into recovered regions
\cite{HastingsSussman2022}.
With the temporal memory law fixed, decreasing \(\beta\) weakens the relative damping of high spatial frequencies, consistent with the narrower fronts also observed in space-fractional FitzHugh--Nagumo models with classical
time derivatives \cite{BuenoOrovioKayBurrage2014,ZhangLiChenZhang2021}.

\begin{figure}[htbp]
\centering
\includegraphics[width=0.30\linewidth]{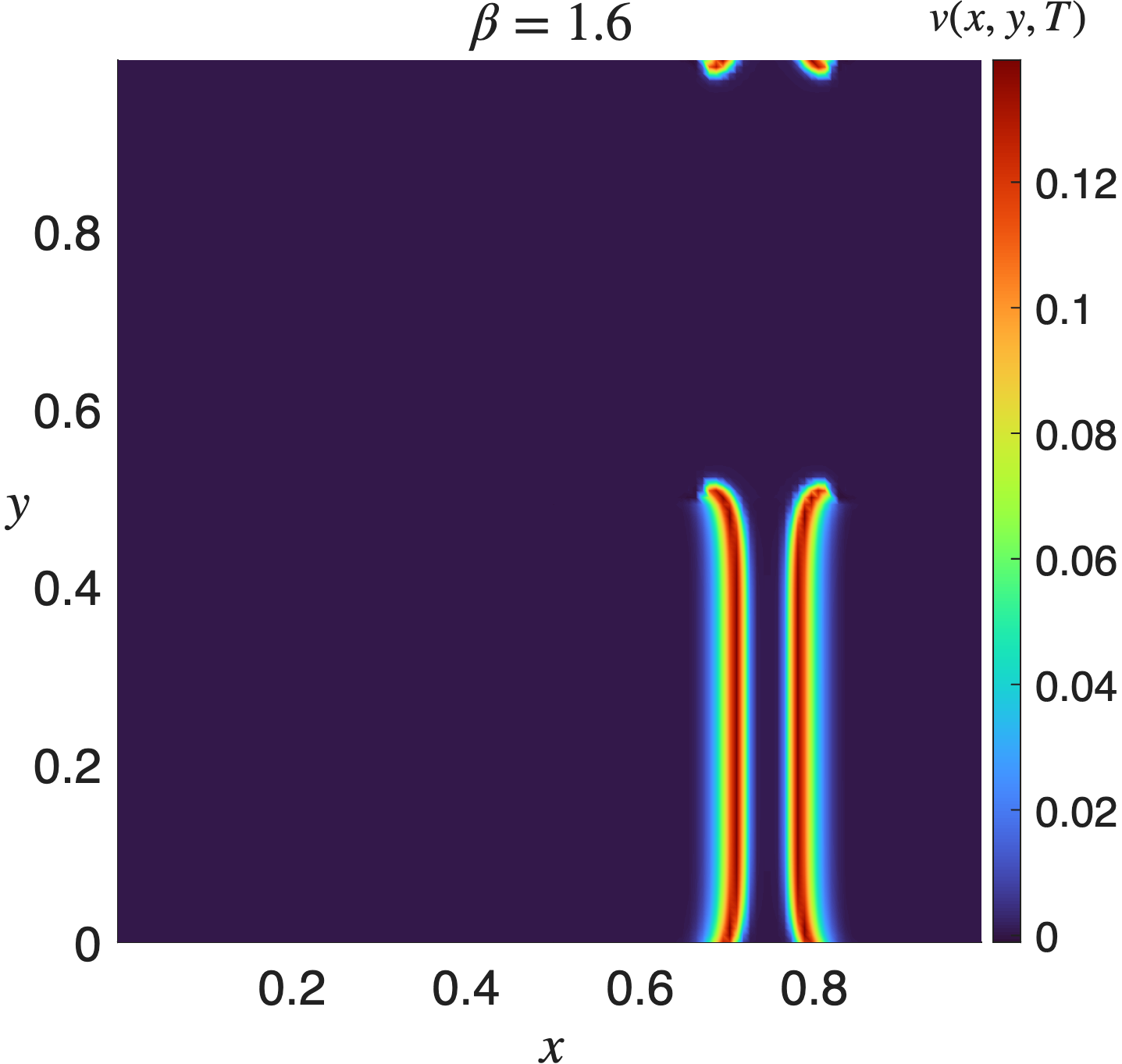}
\hfill
\includegraphics[width=0.30\linewidth]{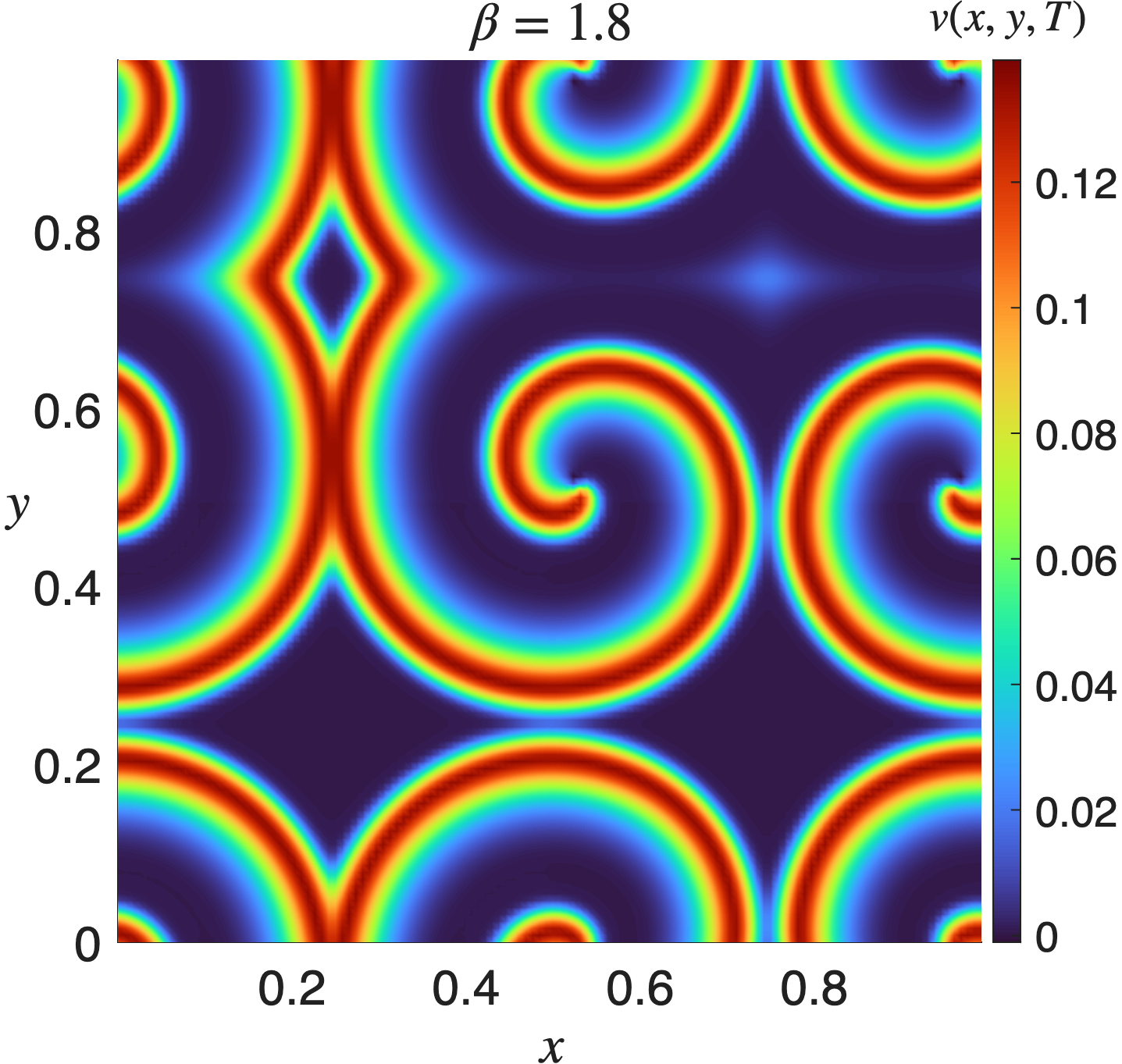}
\hfill
\includegraphics[width=0.30\linewidth]{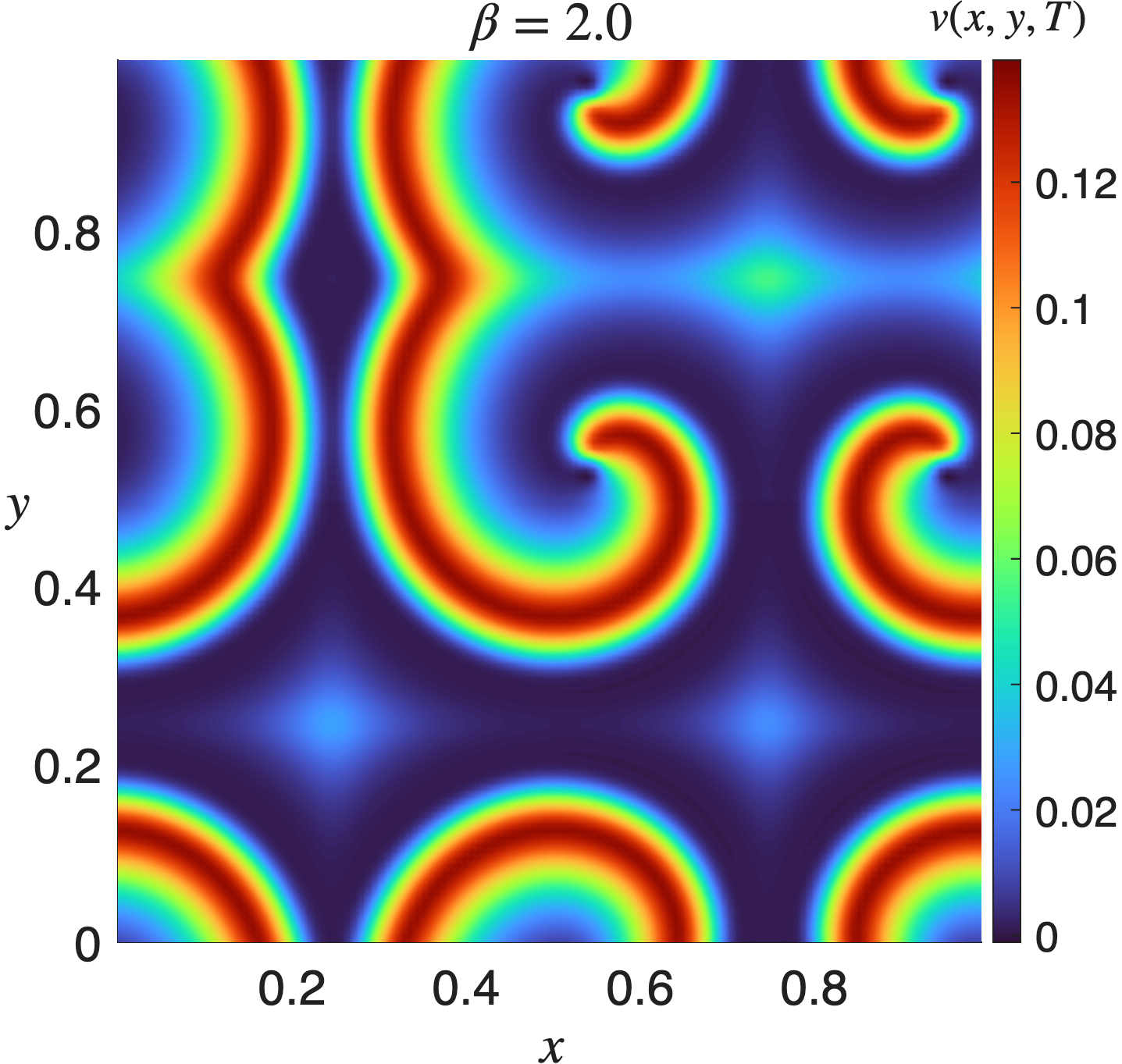}
\caption{Spatial patterns of the FitzHugh--Nagumo recovery field \(v\) at \(T=1000\) for \(\beta=1.6,1.8,2\) (left to right), computed by Block--BE--SOE.}
\label{fig:num-fhn-patterns}
\end{figure}

\begin{table}[!htbp]
\centering

\scriptsize
\setlength{\tabcolsep}{2.2pt}
\caption{PDE compression diagnostics: right-hand-side ratios, normalized bounds and history defects, discrepancies from BE--FH, and restart counts.}
\label{tab:posteriori defect}
\resizebox{\linewidth}{!}{%
\begin{tabular}{@{}lcccccc@{}}
\toprule
Experiment
& \(F_*^{\rm num}/F_{\rm bd}\)
& \(\mathcal C_{\rm des}/\tau_{\rm hist}\)
& \(\mathcal C_{\rm post}/\tau_{\rm hist}\)
& \(\widehat\varepsilon_{\rm hist}/\tau_{\rm hist}\)
& \(D_\infty\)
& restart \\
\midrule
Nagumo, \(\beta=1.6\)
& \(0.2159\) & \(0.8758\) & \(0.1891\)
& \(7.9841\times10^{-3}\) & \(7.1765\times10^{-9}\) & 0\\
Nagumo, \(\beta=1.8\)
& \(0.2051\) & \(0.8758\) & \(0.1797\)
& \(9.6729\times10^{-3}\) & \(5.3478\times10^{-9}\) & 0\\
Nagumo, \(\beta=2\)
& \(0.2378\) & \(0.8758\) & \(0.2083\)
& \(1.1370\times10^{-2}\) & \(4.6938\times10^{-9}\) & 0\\
FitzHugh--Nagumo, \(\beta=1.6\)
& \(0.7627\) & \(0.5358\) & \(0.4086\)
& \(1.3069\times10^{-2}\) & \(1.1563\times10^{-5}\) & 0\\
FitzHugh--Nagumo, \(\beta=1.8\)
& \(0.7456\) & \(0.5358\) & \(0.3995\)
& \(1.6746\times10^{-2}\) & \(4.2705\times10^{-5}\) & 0\\
FitzHugh--Nagumo, \(\beta=2\)
& \(0.7365\) & \(0.5358\) & \(0.3946\)
& \(1.9100\times10^{-2}\) & \(2.1744\times10^{-5}\) & 0\\
\bottomrule
\end{tabular}}
\end{table}

Table~\ref{tab:posteriori defect} verifies the history-defect
bound for all six PDE runs.
The identity
\(\mathcal C_{\rm post}=\frac{F_*^{\rm num}}{F_{\rm bd}}\mathcal C_{\rm des}\)
separates the improvement obtained from the measured right-hand-side norm from the conservatism of the coefficient-residual bound.
The measured history defects are approximately \(3.2\%\)--\(5.5\%\) of \(\mathcal C_{\rm post}\).
Thus the bound controls the observed defects but is not sharp in these tests.

\section{Discussion}

We established the local well-posedness and temporal regularity of reaction–diffusion equations with variable memory based on the Scarpi–Volterra formulation. 
For the numerical approximation, we presented a fully implicit BE–CQ Fourier pseudospectral discretization and designed a blockwise SOE compression approach that controls the history defect through inherited local tolerances and a checked right-hand-side bound. 
We showed that the accelerated scheme retains positive-time first-order accuracy, provided the derived tolerance conditions for weight recovery, compression, and nonlinear solves are satisfied. 
In implementation, reaction-dependent preconditioning and Fourier-transform reuse avoid redundant diffusion transforms without changing the implicit equation. 
Numerical results show close agreement with full-history trajectories and substantially reduced computational costs on finer grids.

Because the compression acts directly on the discrete convolution weights, this construction suggests natural extensions beyond BE--CQ to other convolution-type Volterra discretizations, including those for kernels lacking explicit time-domain formulas. 
Such extensions require accurately computable weights that admit low-rank local SOE approximations, together with stability estimates for the underlying scheme. Although the error analysis accounts for weak initial regularity, the present BE--CQ scheme includes no starting corrections to compensate for reduced accuracy near \(t=0\). 
Higher-order BDF--CQ requires suitable corrections and corresponding stability and error estimates, with the relevant power and possible power--logarithmic terms identified from short-time expansions.
Future work will combine such corrected schemes with defect-controlled SOE compression and investigate hierarchical or semi-infinite-tail recurrences to reduce the \(\mathcal O(N)\) storage footprint of past right-hand-side values  required by the current fixed-block implementation.

\appendix
\section{Proof of Theorem~\ref{thm:semilinear-regularity}}
\label{app:proof-semilinear-regularity}

\begin{lemma}
\label{lem:weakly-singular-gronwall}
Let \(0<\gamma<1\), \(a,b\ge0\), and let \(y\ge0\) belong to \(L^1(0,T)\).
If
\[
y(t)\le a\, t^{\gamma-1}
+b\int_0^t(t-s)^{\gamma-1}y(s)\,ds
\quad\text{for a.e. }t\in(0,T),
\]
then \(y(t)\le C_T a\, t^{\gamma-1}\) almost everywhere, where \(C_T\) depends on \(b,\gamma,T\).
\end{lemma}

\begin{proof}
We iterate the inequality and use \(
(t^{\gamma-1})^{*j}
=\frac{\Gamma(\gamma)^j}{\Gamma(j\gamma)}t^{j\gamma-1},
\) where $(t^{\gamma-1})^{*j}$ denotes the $j$-fold convolution of $t^{\gamma-1}$ with itself.
We deduce that for \(m\gamma\ge1\) the remainder after \(m\) iterations is bounded by
\[
\frac{b^m\Gamma(\gamma)^mT^{m\gamma-1}}{\Gamma(m\gamma)}
\|y\|_{L^1(0,T)}\longrightarrow0,\qquad m\to\infty.
\]
After factoring out \(a\,t^{\gamma-1}\), the resulting series is bounded uniformly on \([0,T]\).
\end{proof}

\begin{proof}[Proof of Theorem~\ref{thm:semilinear-regularity}]

Set \(g_0=\Aa u_0+\Nn(u_0)\).
On the space of continuous functions \(v:(0,T_*]\to X\) with finite norm
\[
\|v\|_{\alpha_1,T_*}
=\sup_{0<t\le T_*}t^{1-\alpha_1}\|v(t)\|_X,
\]
consider
\begin{equation}
\label{eq:derivative-reconstruction}
v(t)=P(t)g_0+
\int_0^tP(t-s)\Nn'\!\left(u_0+\int_0^s v(r)\,dr\right)v(s)\,ds.
\end{equation}
For \(v\) in a fixed ball, \(\|\int_0^s v(r)\,dr\|\le s^{\alpha_1}\|v\|_{\alpha_1,T_*}/\alpha_1\), so the reconstructed values lie in a fixed bounded subset of \(X\).
The estimate \eqref{eq:P-local-bound} and
\begin{equation}
      \label{eq:beta-integral}
t^{1-\alpha_1}\int_0^t
(t-s)^{\alpha_1-1}s^{\alpha_1-1}\,ds
=\frac{\Gamma(\alpha_1)^2}{\Gamma(2\alpha_1)}t^{\alpha_1}
\end{equation}
show that the right-hand side maps a sufficiently large closed ball into itself for small \(T_*\).
Splitting the nonlinear difference as
\[
\Nn'(U_v)(v-w)+[\Nn'(U_v)-\Nn'(U_w)]w,
\qquad U_v(s)=u_0+\int_0^s v(r)\,dr,
\]
gives contraction factors of order \(T_*^{\alpha_1}\) and \(T_*^{2\alpha_1}\), respectively.
Thus a unique fixed point exists.
Integrating \eqref{eq:derivative-reconstruction} and using the chain rule, Fubini's theorem, and \eqref{eq:S-P-identity} shows that \(\widetilde u(t):=u_0+\int_0^t v(s)\,ds\) satisfies \eqref{eq:semilinear-mild-solution}.
Uniqueness in Theorem~\ref{thm:maximal-blowup} gives \(\widetilde u=u\), hence \(v=u'\) on \((0,T_*]\).

To extend this construction from \((0,\tau]\), retain the full history
\[
\mathcal G_\tau(t)=P(t)g_0+
\int_0^\tau P(t-s)\Nn'(u(s))v(s)\,ds,
\qquad t\ge\tau.
\]
This function is continuous near \(t=\tau\).
To see this, split the integral at \(\tau/2\).
On the first part, \(P(t-s)\) is norm continuous away from zero and \(\Nn'(u)v\) is integrable.
On the second part, \(\Nn'(u)v\) is bounded, so continuity follows from \(L^1\)-translation continuity of \(P\).
On \([\tau,\tau+\delta]\), solve
\begin{equation}
\label{eq:derivative-continuation}
v(t)=\mathcal G_\tau(t)+
\int_\tau^tP(t-s)\Nn'\!\left(u(\tau)+
\int_\tau^s v(r)\,dr\right)v(s)\,ds.
\end{equation}
The same estimates in the supremum norm give a contraction for small \(\delta\), since \(\int_0^\delta\|P(s)\|\,ds\to0\).
The combined function satisfies \eqref{eq:derivative-reconstruction}, and its integral is again a mild solution.
Uniqueness therefore identifies the reconstructed function with \(u\).

Let \(T_d\) be the supremum of the endpoints of this reconstruction.
On every reconstructed interval contained in \([0,T]\), boundedness of \(\Nn'\) on \(u([0,T])\) gives
\[
\|v(t)\|\le C_Tt^{\alpha_1-1}
+C_T\int_0^t(t-s)^{\alpha_1-1}\|v(s)\|\,ds.
\]
Lemma~\ref{lem:weakly-singular-gronwall} yields \(\|v(t)\|\le C_Tt^{\alpha_1-1}\), with constants independent of the endpoint. 

If \(T_d<T\), then from \eqref{eq:beta-integral}, we have
\[
\|\mathcal G_\tau(t)\|_X
\le C_T\left(
t^{\alpha_1-1}
+\frac{\Gamma(\alpha_1)^2}{\Gamma(2\alpha_1)}
t^{2\alpha_1-1}
\right)
\]
for \(\tau\le t\le T\).
Since \(t\ge\tau\ge T_d/2>0\), this bound is uniform for \(\tau\in[T_d/2,T_d)\).

On one fixed ball containing the reconstructed values, \(\Nn'\) is bounded and Lipschitz.
The contraction length in \eqref{eq:derivative-continuation} can thus be chosen independently of \(\tau\), extending the reconstruction beyond \(T_d\).
Hence \(T_d\ge T\).
Since \(T<T_{\max}\) is arbitrary, the reconstructed derivative is continuous on \((0,T_{\max})\).
This proves \(u\in W^{1,1}(0,T;X)\) and \eqref{eq:semilinear-derivative-bound}.
Integration gives \(\|u(t)-u(s)\|\le C_T|t-s|^{\alpha_1}\).

Set \(g(t)=\Aa u_0+\Nn(u(t))\).
Then \(g\) is H\"older continuous of order \(\alpha_1\), and \(u(t)-u_0=\int_0^tP(t-s)g(s)\,ds\).
By Lemma~\ref{lem:S-resolvent-identity}, the integral truncated at \(t-\varepsilon\) belongs to \(D(\Aa)\).
Using \(S'=\Aa P\) and then the closedness of \(\Aa\) as \(\varepsilon\downarrow0\), we obtain
\begin{equation}
\label{eq:graph-cancellation}
\Aa(u(t)-u_0)=
\int_0^t\Aa P(t-s)[g(s)-g(t)]\,ds+(S(t)-\Ii)g(t).
\end{equation}
The first integrand is bounded by \(C_T(t-s)^{\alpha_1-1}\), and the constant part of the truncated integral is \((S(t)-S(\varepsilon))g(t)\).
Changing variables to \(r=t-s\) and applying dominated convergence gives continuity of \eqref{eq:graph-cancellation} for \(t>0\).
At zero, the integral is \(O(t^{\alpha_1})\), while
\[
(S(t)-\Ii)g(t)
=(S(t)-\Ii)g(0)+(S(t)-\Ii)[g(t)-g(0)]\longrightarrow0
\]
by strong continuity and local boundedness of \(S\).
Thus \(u\in C([0,T];D(\Aa))\).
The linear equivalence above with \(f=\Nn(u)\) gives \eqref{eq:volterra-nonlinear}, and Proposition~\ref{prop:scarpi-volterra-form} gives the strong equation.
\end{proof}

\bibliographystyle{siam-order}
\bibliography{ref}
\end{document}